\documentclass[12pt]{article}
\usepackage[utf8]{inputenc}
\usepackage{geometry}
\usepackage[numbers]{natbib}

\usepackage{algorithm}
\usepackage{algpseudocode}
\usepackage{amsmath, amsfonts, amssymb, amsthm, graphicx}
\usepackage{authblk}
\usepackage{appendix}
\usepackage{verbatim}
\usepackage{float}
\usepackage{caption}
\usepackage{subcaption}
\usepackage{comment}
\usepackage{xcolor}
\definecolor{myblue}{RGB}{0, 0, 139}
\usepackage[colorlinks, citecolor=myblue, linkcolor=myblue]{hyperref}

\usepackage{setspace}
\usepackage{parskip}

\newtheorem{proposition}{Proposition}
\newtheorem{theorem}{Theorem}
\newtheorem{lemma}{Lemma}
\newtheorem{remark}{Remark}

\newtheorem{assumption}{Assumption}

\newtheorem{example}{Example}
\newtheorem{subassumption}{Assumption\normalfont}[assumption]

\newenvironment{assumption*}
 {\ifnum\value{subassumption}=0 \stepcounter{assumption}\fi\subassumption}
 {\endsubassumption}
\newenvironment{assumption+}[1]
 {\renewcommand{\thesubassumption}{#1}\subassumption}
 {\endsubassumption}

\newcommand{\R}{\mathbb{R}}

\let\temp\phi
\let\phi\varphi
\let\varphi\temp

\title{Explicit Constraints on the Geometric Rate of Convergence of Random Walk Metropolis-Hastings}

\author{Riddhiman Bhattacharya\thanks{Purdue University, bhatta76@purdue.edu} ~ and Galin L. Jones\thanks{School of Statistics, University of Minnesota, galin@umn.edu}}    

\begin{document}

\maketitle

\begin{abstract}
Convergence rate analyses of random walk Metropolis-Hastings Markov chains on general state spaces have largely focused on establishing sufficient conditions for geometric ergodicity or on analysis of mixing times.  Geometric ergodicity is a key sufficient condition for the Markov chain Central Limit Theorem and allows rigorous approaches to assessing Monte Carlo error.  The sufficient conditions for geometric ergodicity of the random walk Metropolis-Hastings Markov chain are refined and extended, which allows the analysis of previously inaccessible settings such as Bayesian Poisson regression.  

The key technical innovation is the development of explicit drift and minorization conditions for random walk Metropolis-Hastings, which allows explicit upper and lower bounds on the geometric rate of convergence. Further, lower bounds on the geometric rate of convergence are also developed using spectral theory.  
The existing sufficient conditions for geometric ergodicity, to date, have not provided explicit constraints on the rate of geometric rate of convergence because the method used only implies the existence of drift and minorization conditions.  

The theoretical results are applied to random walk Metropolis-Hastings algorithms for a class of exponential families and generalized linear models that address Bayesian Regression problems.  
\end{abstract}

\section{Introduction}
\label{sec:intro}
The complicated distributions that arise in modern statistics often require Markov chain Monte Carlo (MCMC) methods in order to use them for statistical inference \cite{brooks:etal:2011}.  Metropolis-Hastings (MH) algorithms \cite{hast:1970, metr:1953} are fundamental to MCMC and are often used as building blocks for component-wise MCMC algorithms \cite{john:jone:neat:2013, jone:robe:rose:2014} when full-dimensional updates are infeasible.  Under weak regularity conditions MH will eventually produce a representative sample from the desired target distribution, but a fast rate convergence is crucial to the efficacy and reliability of the algorithm in practice.

Assessing the reliability of a MH simulation is possible when a Markov chain CLT holds \cite{fleg:etal:2020, geye:1992, jone:etal:2006, robe:etal:viz:2021, vats:etal:sve:2018}.  A key sufficient condition for a Markov chain CLT is geometric ergodicity \cite{chan:geye:1994, geye:1994, jone:2004}, which will also ensure that the Markov chain will converge quickly, at least in a qualitative sense \cite{jone:hobe:2001}. As such, there has been significant work on MH Markov chains on general state spaces focused on establishing sufficient conditions for geometric ergodicity \citep{brofos2023geometric, corn:etal:2019, john:geye:2012, roy2021convergence, tier:1994}.  

In particular, there are well-known sufficient conditions for geometric convergence of random walk MH (RWMH) \cite{jarn:hans:2000, meng:twee:1996, robe:twee:1996}, which are now described informally but will be stated more carefully later.  Suppose the target distribution $F$ has a Lebesgue density $f$ on $\R^p$, then the conditions for geometric ergodicity come in two parts \cite[][Theorem 4.3]{jarn:hans:2000}.  The first is that $f$ is superexponentially light in the sense that the tails of $f$ decay at an exponential rate.   If $\|\cdot\|$ is the usual Euclidean norm, the second is the so-called curvature condition  
\begin{align}\label{eq:curvature}
\limsup_{\|x\|\to \infty}\left<\frac{x}{\|x\|},\frac{\nabla f(x)}{\|\nabla f(x)\|}\right> < 0.
\end{align}
 The curvature condition has been verified in some applications \citep{chri:moll:waag:2001, vats:etal:moa:2019}, but it does not hold for some important settings such as Bayesian Poisson regression models \cite{john:geye:2012}. This is addressed in Section~\ref{sec:GE_RWMH} where a more general  geometric structure of $f$ is exploited to refine and generalize the sufficient conditions for geometric ergodicity of RWMH.  The new conditions are verified for an important class of exponential family Bayesian regression models, including some Bayesian Poisson regression models. The details are given in Theorems~\ref{thm:rwmh_ge} and~\ref{expfam:ergdocitythm}.  

A limitation of the sufficient conditions, including those developed here,  is that they only imply the {\em qualitative} property of geometric ergodicity and previously have not provided any constraints in the geometric rate of convergence \citep[][p. 342]{jarn:hans:2000}.  Some additional notation is required to describe this more precisely.  If $d_{TV}$ denotes total variation distance and $\mathcal{L}(X_n)$ the marginal distribution of the $n$-th step of the RWMH, then the RWMH Markov chain is geometrically ergodic for an initialization at $x \in \mathbb{R}^{p}$ if there exists a real-valued function $M$ on $\mathbb{R}^{p}$ and $0< \rho < 1$ such that 
\begin{equation}
\label{eq:geom_erg}
d_{TV}(\mathcal{L}(X_n), F)  \le M(x) \rho^n .
\end{equation}
The sufficient conditions  only imply the existence of an appropriate $M$ and $\rho$ because the technique used in deriving them only implies the existence of appropriate drift and minorization conditions.  This is addressed in Section~\ref{sec:main_results} by deriving \textit{explicit} drift and minorization conditions for  RWMH, which appears to be a first for MH on general state spaces and requires some novel technical arguments. This will then allow an appeal to existing upper and lower bounds \cite[see, e.g.,][]{brow:jone:2022a, rose:1995a} for $M$ and $\rho$.  The details are given in Theorem~\ref{thm:main_result} and Proposition~\ref{prop1}.

While establishing drift and minorization conditions is a standard approach to establishing geometric ergodicity \cite{meyn:twee:2009} another approach entails establishing a spectral gap for the Markov kernel operator of the Markov Chain~\cite{baxe:2005}. Both approaches often provide similar bounds on the convergence of the Markov chain~\cite{roberts_tweedie_2001}.   Lower bounds on $\rho$ using spectral theory are developed in Section~\ref{sec:spectral_results}.  The details are given in Propositions~\ref{Spectral:lwr:bnd1} and~\ref{Spectral:lwr:bnd2}.

Due to the importance of RWMH, there is a substantial amount of related work~\citep{andrieu2022explicit, Belloni_2009, john:smit:2018, wu2022minimax, zhou2022dimension}. For example, some nice recent work on explicit convergence bounds using spectral methods~\cite{andrieu2022explicit}  assumes a Gaussian proposal, strong log-concavity, and Lipschitz gradient.  By contrast, our work does not require distributional assumptions on the proposal and the target log-likelihood need not be strongly concave.  Others have focused on bounding the conductance, restricted to a compact subset of $\mathbb{R}^p$ \cite{Belloni_2009, dwivedi2018log}. There is also work on RWMH mixing time bounds which under the assumption that a drift condition exists~\citep{john:smit:2018}.  Other work considers RWMH in specific settigns.  For example, in a Bayesian variable selection setting \citep{zhou2022dimension} take a different approach by considering a two-step drift condition or bounds on the acceptance probability. 



The remainder is arranged as follows.  In Section~\ref{sec:GE_RWMH} the RWMH algorithm considered is defined and new sufficient conditions for geometric ergodicity are developed. 
Section~\ref{sec:main_results}, considers upper and lower bounds for the geometric rate of convergence of RWMH. In Section~\ref{glm_type_families}, the results are applied to Bayesian generalized linear models and in Section~\ref{sec:spectral_results} spectral theory is used to establish lower bounds on the geometric rate of convergence of RWMH.

\section{Geometric Ergodocity of RWMH}\label{sec:GE_RWMH}

Suppose the distribution $F$ with support $\R^p$ has Lebesgue density
$f$ and for each $x \in \mathbb{R}^p$, the proposal distribution
$Q(x ,\cdot)$ has Lebesgue density $q(x, \cdot)$ such that
$q(x, y) = q(y, x)=q(\|x-y\|)>0$ for all $x, y \in \mathbb{R}^p$ and
$f(x)>0$ for all $x \in \mathbb{R}^p$.  The Hastings ratio is
\[
h(x, y) = \frac{f(y)}{f(x)}
\]
and set $\alpha(x,y)=h(x,y) \wedge 1$. Let $\delta_x(\cdot)$ denote the Dirac measure at $x$.  The one-step Markov transition kernel for the RWMH is given by
\begin{align}\label{RWMH:kernel}
    P(x,dy)=\alpha(x,y)q(x,y)dy + \delta_x(dy) \left[\int_{\mathbb{R}^p} \left(1-\alpha(x,y)\right) q(x,y) dy \right]
\end{align}
and for $n \ge 2$ 
\[
P^n (x, dy) = \int_{\mathbb{R}^p} P(x, du) P^{n-1}(u, dy).
\]
The assumptions imply that the Markov kernel has $F$ as its invariant distribution, is aperiodic, $F$-irreducible, and Harris recurrent.  Hence if $\|\cdot \|_{TV}$ denotes the
total variation norm, for every $x \in \mathbb{R}^{p}$, as $n \to \infty$,
\[
\| P^n(x, \cdot) - F(\cdot) \|_{TV} \to 0.
\]

The remainder of this section focuses on developing sufficient
conditions for the geometric ergodicity of RWMH chains.  However,
these conditions will not provide any constraints on $\rho$, which is
addressed in Section~\ref{sec:main_results}. 

\begin{assumption}
\label{assm:super}
    Let $\langle \cdot , \cdot \rangle$ denote the usual Euclidean inner
    product.   The target density $f(x)$ is continuously differentiable and
    satisfies the superexponential condition  
\begin{align}\label{superexp}
\limsup_{\|x\|\to \infty}\left\langle\frac{x}{\|x\|},\nabla \log f(x)\right\rangle=-\infty .
\end{align}
\end{assumption}

Note that the region of certain acceptance for RWMH is given by 
\[
A(x) =\left\{y: f(y) \ge f(x)\right\}.
\]
 It is crucial that
$A(x)$ have a convenient geometric structure.  Suppose
$0<\alpha<\cos^{-1}(7/8)$ and let
$\tilde{\epsilon}_{\alpha} = (8-8\cos\alpha)^{1/2}$. Denote the unit
hollow sphere centered at the origin by $\mathbb{S}^{p-1}$. For
\begin{align}\label{eps:alpha:defn}
    \epsilon_{\alpha}<\tilde{\epsilon}_{\alpha},
\end{align} define the cone
\begin{equation}
\label{eq:cone}
W^K_{\epsilon_{\alpha}}(x)=\left\{x-a\xi:0<a<K, \ \xi \in \mathbb{S}^{p-1}, \ \left\|\xi-\frac{x}{\|x\|}\right\|\le \frac{\epsilon_{\alpha}}{2} \right\}.
\end{equation}

Let $\overline{A(x)}$ denote the closure of $A(x)$ and let
$\partial\overline{A(x)}$ denote its boundary.

\begin{assumption}\label{assm:cone}
There exists an $M_p>0$ such that for any $\| x \| > M_p$, there is $y \in \partial\overline{A(x)}$, so that $W^K_{\epsilon_{\alpha}}(y) \subset A(x)$. 
\end{assumption} 

The following result shows that these assumptions are sufficient for
ensuring geometric ergodicity of the Random Walk Metropolis-Hastings
algorithm. 

\begin{theorem}
\label{thm:rwmh_ge}
If Assumptions~\ref{assm:super} and ~\ref{assm:cone} hold, then the
RWMH algorithm is geometrically ergodic. 
\end{theorem}
\begin{proof}
     A proof is provided in Appendix~\ref{app:GE_RWMH}, but the
argument is essentially a refinement of previous work \cite[][Theorem
4.3]{jarn:hans:2000}.
\end{proof}
Theorem~\ref{thm:rwmh_ge} extends prior results~\citep{jarn:hans:2000} by considering a more general assumptions. Directly checking Assumption~\ref{assm:cone} can be challenging. There are sufficient conditions which may be easier to verify in some applications. 

\begin{proposition}
\label{prop:key}
Let Assumption~\ref{assm:super} hold.  If (i) the
curvature condition~\eqref{eq:curvature} holds or (ii) $f$ is
log-concave, then Assumption~\ref{assm:cone} holds.
\end{proposition}
\begin{proof}
         See Appendix~\ref{app:key}.
\end{proof}

The following is a stronger version of Assumption~\ref{assm:super}, which will be used below to establish Assumption~\ref{assm:cone} with an explicit value for $M_p$  and $\epsilon_{\alpha}$.
\begin{assumption}
\label{assm:super'}
Suppose $f_s(\cdot)$ is continuous and strictly increasing and $C_1 \in \mathbb{R}$ is such that there exists $M_s\ge 0$ so that for any $x\in \mathbb{R}^p$ with $\| x\| >M_s$, 
\begin{align}\label{super:rate}
\left\langle \frac{x}{\Vert x\Vert}, \nabla \log f(x)\right\rangle \le C_1-f_s(\Vert x\Vert).
\end{align}
\end{assumption}

\begin{proposition}
\label{prop:curvature1}
Let Assumption~\ref{assm:super'} hold. If there exists $M'_p>0$ and $\eta \in (0,1)$ such that for any $x\in \mathbb{R}^p$ with $\| x \| > M'_p$, 
\begin{align*}
f_s(\Vert x\Vert)-C_1\ge \eta \Vert \nabla \log f(x)\Vert ,
\end{align*}  
then Assumption~\ref{assm:cone} is satisfied with $M_p=M'_p$ and $\epsilon_{\alpha}=\eta$.
\end{proposition}
\begin{proof}
    See Appendix~\ref{app:curvature1}.
\end{proof}
\subsection{Application to Exponential Family Models}
Suppose $X_1,X_2,\ldots,X_n$ are conditionally independent given $\theta$ with density  characterized by
\[
p(x_i\mid \theta) \propto \exp\left\{\langle x_i, \theta \rangle -c(\theta)\right\} .
\]
A function $g(\theta)$ is $m$-strongly convex if $\nabla^2 g(\theta)\ge m\, I$ for all $\theta \in \mathbb{R}^p$ and some $m > 0$. Assume the prior on $\theta$ is
\[
\nu(\theta)\propto \exp\left\{ -g(\theta)\right\}
\]
with $g(\theta)$ strongly convex.
Note that $\nu(\theta)$ is proper due to the strong convexity of $g(\theta)$.
Letting $X = \{ X_1,X_2,\ldots,X_n \}$, the posterior is characterized by
\begin{equation}
\label{expfamposterior}
f(\theta \mid X ) \propto \exp \left \{ \sum_{i=1}^{n} \langle x_i, \theta \rangle - n c(\theta)-g(\theta) \right\} .
\end{equation}

\begin{theorem}\label{expfam:ergdocitythm}
    The RWMH algorithm with stationary distribution defined by the density $f(\theta \mid X)$ is geometrically ergodic.    
\end{theorem}
\begin{proof}
The proof of Theorem~\ref{expfam:ergdocitythm} follows by checking the conditions of Theorem~\ref{thm:rwmh_ge} and Proposition~\ref{prop:key}, which is done in Appendix~\ref{app:proof of expfam:ergdocitythm}.
\end{proof}

\section{Constraining the Geometric Rate}
\label{sec:main_results}

Theorem~\ref{thm:rwmh_ge} and other results on geometric ergodicity of multidimensional RWMH \cite{jarn:hans:2000, meng:twee:1996, robe:twee:1996} have focused on establishing the existence of $\rho$, but provide no constraints on its value, which is the focus here. Two additional assumptions are required to proceed.

\begin{assumption}
\label{assm:envelope}
    There exists a function $\tilde f:\mathbb{R}_+\to \mathbb{R}_{+}$, non-decreasing, such that, for all $x$,
    \begin{align*}
        \left|\log f(x)\right| \le \tilde f(\left\|x\right\|).
    \end{align*}
\end{assumption}


\begin{assumption}
\label{assm:prop}
    The proposal density satisfies $q(\|x\|) < q(\|y\|)$ when $\|x\|>\|y\|$.
\end{assumption}

Some preliminary notation is required in order to state the main result of this section cleanly.  Recall that $\epsilon_{\alpha}$ was previously defined~\eqref{eq:cone}.  If $K < 1/3$ and 
\begin{align}\label{eq:geomradius}
R_{\alpha}=K\left(1-\epsilon_{\alpha}^2/8\right)\left[1-\left(1-\epsilon_{\alpha}^2/8\right)^2\right]^{1/2}\left\{1+\left[1-\left(1-\epsilon_{\alpha}^2/8\right)^2\right]^{1/2}\right\}^{-1},
\end{align}
define 
\begin{align}\label{eq:driftcoeff}
\tilde{\lambda}=1-\frac{1}{1+\epsilon_{\alpha}}Q\left(0,\left[-\frac{R_{\alpha}}{\sqrt{p}}-1,\frac{R_{\alpha}}{\sqrt{p}}-1\right]\times \left[-\frac{R_{\alpha}}{\sqrt{p}},\frac{R_{\alpha}}{\sqrt{p}}\right]^{p-1}\right),
\end{align}
and  
let \(C_B(p)\) denote the volume of the unit ball $\{x:\|x\|\le 1\}$ of dimension $p$.

Let
\begin{align}\label{delta:defn}
    \delta <\frac{\epsilon}{q(0)\,C_B(p)\, K^{p-1}_{\epsilon}},
\end{align}
 $p^*=\max f(x)$, $M^*=\max\{M_p,M_s\}$,
 \begin{align}\label{eps_def_1}
    \epsilon\le \frac{1}{8}\frac{\epsilon_{\alpha}}{1+\epsilon_{\alpha}}\,Q\left(0,\left[-\frac{R_{\alpha}}{\sqrt{p}}-1,\frac{R_{\alpha}}{\sqrt{p}}-1\right]\times \left[-\frac{R_{\alpha}}{\sqrt{p}},\frac{R_{\alpha}}{\sqrt{p}}\right]^{p-1}\right),
\end{align}
and
\begin{equation}\label{R:eps:def_1}
\begin{split}
\mathcal{R}_{\epsilon}&=\max\left\{f_s^{-1}\left(C_1-\frac{1}{\delta}\log\epsilon\right)+K_{\epsilon},\ 2K^2_{\epsilon}\left[\left(\frac{\epsilon}{q(0)\, C_B(p) \delta}\right)^{1/p-1} -K_{\epsilon}\right]^{-1}+K_{\epsilon}, \right. \\
& \left. \quad \quad \quad f_s^{-1}\left(C_1-\log \left(\exp\left\{-\tilde f\left(f_s^{-1}\left(C_1\right)\vee M^*\right)\right\}/p^*\wedge 1\right)\right)+1, \ M^*+K_{\epsilon}\vee 1\right\}.
\end{split}
\end{equation}
Set
\begin{align}\label{drift:const}
    b=3\exp\{\tilde{f}(\mathcal{R}_{\epsilon})/2\},
\end{align}
\begin{equation}\label{A:defn}
	 A=1+\frac{4\tilde\lambda b+2\tilde\lambda\epsilon_{\alpha}}{1-\tilde{\lambda}}+2b=\frac{2\tilde\lambda\left(b+\epsilon_{\alpha}\right)+2b}{1-\tilde{\lambda}},
\end{equation}
\begin{equation}\label{tilde_alpha:defn}
 \tilde\alpha=\left(1+\frac{2b+\epsilon_{\alpha}}{1-\tilde{\lambda}}\right)\left[1+2b+\frac{\tilde{\lambda}}{1-\tilde{\lambda}}\left(2b+\epsilon_{\alpha}\right)\right]^{-1},
\end{equation}
and
\begin{align}\label{R_1:def_1}
  \mathcal{R}_{\max}=f_s^{-1}\left(C_1-\log\left(\frac{(1-\tilde\lambda)^2}{2\,p^* \,(2b+\epsilon_{\alpha})^2}\wedge 1\right)\right)\vee M^* .
\end{align}

\begin{theorem}
\label{thm:main_result}
Suppose Assumptions~\ref{assm:cone}--\ref{assm:prop} hold. Then for any $\epsilon$ as defined in \eqref{eps_def_1}, 
\[
\tilde{\eta}=\exp\{-2 \,\tilde{f}(\mathcal{R}_{\max})\} q(\mathcal{R}_{\max}) \,\mu^{Leb}(B(0,\mathcal{R}_{\max})),
\]
and
 $b=3\exp\{\tilde{f}(\mathcal{R}_{\epsilon})/2\}$, the RWMH is $(M, \rho)$-geometrically ergodic; where for all $x \in \mathbb{R}^p$,
\begin{equation*}
M(x) \le \left(2+\frac{b}{1-\tilde \lambda} + \frac{1}{\sqrt{f(x)}}\right)
\end{equation*}
and for $0 < r < 1$
\begin{equation}
\label{eq:constraints}
1 -  \inf_{x \in \mathbb{R}^p} \left[ \int_{\mathbb{R}^p} \alpha(x,y) q(x,y) dy\right] \le \rho \le  \max\left\{\left(1-\tilde\eta\right)^{r},\tilde \alpha^{-(1-r)} A^r\right\}. 
\end{equation}
\end{theorem}
\begin{proof}
   The lower bound follows from recent work \cite[][Theorem 2]{brow:jone:2022a}. The upper bound follows from establishing explicit drift and minorization conditions with $V(x)=f(x)^{-1/2}$ as the drift function and $\delta_{x}(\cdot)$ as the starting measure, see Propositions~\ref{prop:driftJH}--\ref{prop:R_eps}, and applying the result of existing general upper bounds \citep[Theorem 12]{rose:1995a}\label{rose:thm:1995a}.
\end{proof}

The lower bound in equation \eqref{eq:constraints} is general, in part because it can be difficult to calculate it explicitly.  However, it often suffices to bound it or approximate it with standard Monte Carlo  methods.

\begin{proposition}
 \label{prop:more_lowerbounds}
 Suppose the RWMH algorithm is $(M, \rho)$-geometrically ergodic.  If there exists $B : \mathbb{R}^{p} \to (0,\infty)$ such that $q(x,y) \le B (x)$, then, for each $x \in \mathbb{R}^{p}$,
 \begin{equation}
\label{eq:more_constraints}
1 - \frac{B(x)}{f(x)} \le \rho.
\end{equation}
Alternatively, if $f$ has at most countably many modes, then for mode $x^*$,
\begin{equation}
\label{eq:even_more_constraints}
1 - \frac{1}{f(x^*)} \int_{\mathbb{R}^p} f(y) q(x^*, y) dy \le \rho.
\end{equation}
\end{proposition}
\begin{proof}
         The proof is given in Appendix~\ref{sec:more_lb}.
\end{proof}
\begin{example}
Suppose, for $i=1,\ldots,n$,
\begin{equation*}
Y_i \mid X_i=x_i \stackrel{ind}{\sim} Poi(e^{x^{\mathsf{T}}_i\theta})
\end{equation*}
and if $g(\theta)$ is strongly convex, the prior satisfies $\nu(\theta) \propto  e^{-g(\theta)}$.
The posterior is characterized by
\begin{align}\label{lklihdpois}
    f(\theta \mid y,x) \propto \exp\left\{ \sum_{i=1}^{n} \left(y_i x_i^{\mathsf{T}}\theta-e^{x_i^{\mathsf{T}}\theta} \right) -g(\theta) \right\}.
\end{align} 

Denote by $c(\theta,x_i)=\exp\left(x^{\mathsf{T}}_i\theta\right)$. Notice that
\[
    \nabla^2 c(\theta,x_i)=e^{x_i^{\mathsf{T}}\theta} \, x_ix_i^{\mathsf{T}} \ge 0,
\]
which implies  $c(\theta,x_i)$ is convex. Also, for some for some $m > 0$ 
\[
-\nabla^2 \log f(\theta\mid y,x)=\nabla^2 g(\theta) +\sum_{i=1}^{n} e^{x^{\mathsf{T}}_i\theta}x_ix^{\mathsf{T}}_i >m\,I 
\]
and hence $-\log f(\theta\mid y,x)$ is strongly convex. Thus there exists a unique global minimum, say $\theta^*(y,x)$, which may be computed numerically.

The same argument as the in the proof of Theorem~\ref{expfam:ergdocitythm} will establish that the RWMH chain is geometrically ergodic and Proposition~\ref{prop:more_lowerbounds} will yield lower bounds on the geometric rate of convergence. Specifically, for a bounded proposal density $q(x,y) \le B$ and
\[
1-\frac{B}{f(\theta^*(y,x)\mid y,x)}\le \rho.
\]
As a concrete example suppose the proposal is a Gaussian centered at the current state $x$ with covariance matrix $hI$.  Then
\[
q(x,y) \le \frac{1}{(2 \pi h)^{p/2}}
\]
and hence
\[
1-\frac{1}{f(\theta^*(y,x)\mid y,x) (2 \pi h)^{p/2}}\le \rho.
\]
This suggests choosing the scale $h < 1$ since otherwise it could easily be the case that $\rho \approx 1$ and, even though it is geometrically ergodic, the convergence would be prohibitively slow.
\end{example}

\subsection{Combining Densities}

The target density function $f$ may be the result of combining two other density functions through multiplication or mixing.  More specifically, suppose $f_1$ and $f_2$ are density functions and either $f=f_1 f_2$ such that
\[
\int_{\mathbb{R}^{p}} f_1(x)  f_2 (x) dx =1
\]
or $f=a_1 f_1+a_2f_2$ where $a_1,a_2$ are non-negative real numbers such that $a_1+a_2=1$. It can be convenient to study the convergence properties of RWMH having invariant density $f$ through the properties of $f_1$ and $f_2$. The following result gives conditions for this.  A similar result appeared previously, but was stated without proof under stronger assumptions \cite[][Theorem 4.4]{jarn:hans:2000}. 

\begin{proposition}
\label{prop:class}
Suppose for $i=1,2$, there exists $\tilde{f}_{i}$ so that the density $f_i$ satisfies Assumption~\ref{assm:envelope} and there exists $f_{is}$ and $C^{*}_{i}$ so that $f_i$ satisfies Assumption~\ref{assm:super'}.  In addition, assume that for $0< \eta_i <1$ there is $M^*_{i} > 0$ such that for $|x| > M^*_{i}$ 
\begin{equation}
\label{eq:curvature:exact}
\left\langle\frac{x}{\Vert x\Vert}, \frac{\nabla f_i(x)}{\Vert\nabla f_i(x)\Vert} \right\rangle \le -\eta_i .
\end{equation}
\begin{enumerate}
\item If $f = f_1 f_2$, then $f$ satisfies Assumption~\ref{assm:cone} with $M^*=M^*_{1}\vee M^*_{2}$, Assumption~\ref{assm:envelope} with  $\tilde{f}=\tilde{f}_1+\tilde{f}_2$, and Assumption~\ref{assm:super'} with  $f_s=f_{1s}+f_{2s}$ and $C_1=C^*_1+C^*_2$, and $\eta=\eta_1\wedge\eta_2$.

\item If $f=a_1 f_1 + a_2 f_2$ then $f$ satisfies Assumption~\ref{assm:cone} with $M^*=M^*_{1}\vee M^*_{2}$, Assumption~\ref{assm:envelope} with  $\tilde{f} = \log(\exp(\tilde{f}_{1} + \tilde{f}_{2})+1)$, and Assumption~\ref{assm:super'} with  $f_s = f_{1s} \wedge f_{2s}$ and $C_1 = C^*_1\vee C^*_2$, and  $\eta=\eta_1\wedge\eta_2$.
\end{enumerate}
\end{proposition}

\begin{proof}
    See Appendix~\ref{app:prop_class}.
\end{proof}

If the proposal satisfies Assumption~\ref{assm:prop}, then, under the conditions of Proposition~\ref{prop:class}, the RWMH is geometrically ergodic and computable bounds are available.  This is illustrated in the following example which is an extension of previous work~\cite[][Example 5.2]{jarn:hans:2000}. 

\begin{example}
For $a > 0$, let $f_1(x,y) = \sqrt{a} \pi^{-1}\exp\left(-a\,x^2-y^2\right)$ and $f_2(x,y)= \sqrt{a} \pi^{-1} \exp\left(-x^2-a\,y^2\right)$. Set
\[
f(x,y) =  \frac{\sqrt{a}}{2 \pi} \exp\left(-a\,x^2-y^2\right) + \frac{\sqrt{a}}{2 \pi} \exp\left(-x^2-a\,y^2\right).
\]
For $z \in \mathbb{R}_{+}$ define $\tilde{f}_{1} (z) = \max\{a,1\} z^2 + \sqrt{a}/\pi$.  Note that 
\begin{equation}
\label{mixture:gaussian:eqn1}
\left|\log f_1(x,y)\right| \le \max\{a,1\}\left( x^2 + y^2 \right) + \frac{\sqrt{a}}{\pi} = \tilde{f}_1(\|(x,y)\|)
\end{equation}
and, similarly, $\left|\log f_2(x,y)\right| \le \tilde{f}_1(\|(x,y)\|)$.  Hence set $\tilde{f}_1(\|(x,y)\|) = \tilde{f}_2(\|(x,y)\|)$.

Next, $f_1$ and $f_2$ satisfy Assumption~\ref{assm:super'} with   $f_{1s}(z) = f_{2s}(z) = 2 \min\{a,1\} z$ and $C^*_1 = C^*_2 = 0$ since
\[
\left\langle \frac{(x, y)}{\| (x, y) \|}, \nabla \log f_1(x,y)\right\rangle  = \left\langle \frac{\left(x, y \right)}{\left(x^2+y^2\right)^{1/2}},\left(-2a\, x, -2\,y\right)\right\rangle\le -2\,\min\{a,1\} \left(x^2+y^2\right)^{1/2}\]
and similarly for $f_2$.
To see that $\eta_1=\min\{a,1\}$ in \eqref{eq:curvature:exact} notice
\begin{align*}
\left\langle\frac{(x,y)}{\| (x,y) \|}, \frac{\nabla f_1(x,y)}{\|\nabla f_1(x,y)\|} \right\rangle 
 & = \left\langle \frac{\left(x,y\right)}{\left(x^2+y^2\right)^{1/2}}, \frac{\left(-2a\,x,-2y\right)}{\left(4a^2 x^2+4y^2\right)^{1/2}}\right\rangle \\
 & \le -\frac{\min\{a,1\}\left(x^2 + y^2\right)}{\max\{a^2 , 1\}\left(x^2 + y^2\right)} \\
 & \le  -\min\{a,1\}.
\end{align*}
Similarly, $\eta_2=\min\{a,1\}$. 

Consider any $\delta$ as in \eqref{delta:defn}.
In this given setting, $f_s(x)=2\, \min\{a,1\}\,x$, with $M^*=0$, $C_1=0$ and  $\tilde{f}$ is as defined in \eqref{mixture:gaussian:eqn1}. 
 Therefore 
\begin{align*}
\mathcal{R}_{\epsilon}&=\max\Bigg\{-\frac{1}{2\min\{a,1\}\,\delta}\log\epsilon+K_{\epsilon},\ 2K^2_{\epsilon}\left[\left(\frac{\epsilon}{q(0)\, C_B(p) \delta}\right)^{1/p-1} -K_{\epsilon}\right]^{-1}+K_{\epsilon},\\
& \quad \quad \quad \frac{1}{2\min\{a,1\}}\left(\bar{C}-\log \left(p^*\wedge 1\right)\right)+1\Bigg\}
\end{align*}
and $b=\exp(3\,\tilde{f}(\mathcal{R}_{\epsilon}))$.
Also for $\epsilon_{\alpha}<\min\{a,1,\tilde{\epsilon}_{\alpha}\}$ as in \eqref{eps:alpha:defn} define $\tilde{\lambda}$ as in \eqref{eq:driftcoeff}. Using this,
\[\mathcal{R}_{\max}=\frac{1}{2\min\{a,1\}}\left(-\log\left(\frac{(1-\tilde\lambda)^2}{2\,p^* \,(2b+\epsilon_{\alpha})^2}\wedge 1\right)\right).\]
Note that both $\mathcal{R}_{\epsilon}$ and $\mathcal{R}_{\max}$, require $p^*$, which may be obtained numerically. 
\end{example}

\section{Application to Bayesian Generalized Linear Models} \label{glm_type_families}

 Let $\Pi:\mathbb{R}^p \times \mathbb{R}^p \to \mathbb{R}^m$ , $T: \mathbb{R}^p \to \mathbb{R}^m$ for some $m\ge 1$, and $c: \mathbb{R}^p\times \mathbb{R}^p\to \mathbb{R}$.  Suppose $(Y_i,X_i)_{i=1}^{n}$ are random variables and assume that the $Y_{i} \mid X_i, \theta$ are conditionally independent  with probability function
 \begin{equation}
     \label{glmfam}
     p(y\mid \theta, x) \propto \exp \left(\left\langle\Pi(\theta,x), T(y)\right\rangle -c(\theta,x)\right).
 \end{equation}
If $g: \mathbb{R}^p \to \mathbb{R}$, assume a prior measure
 \begin{equation}
    \label{glmprior} 
    \exp\left(-g(\theta)\right) d\theta
 \end{equation}
 and assume that the posterior exists and is characterized by 
 \begin{align}\label{glmposterior}
f(\theta\mid y,x)\propto \exp\left(\sum_{i=1}^{n}\left\langle\Pi(\theta,x_i), T(y_i)\right\rangle-\sum_{i=1}^{n}c(\theta,x_i)-g(\theta)\right).
\end{align}

Consider a RWMH chain with proposal satisfying assumption~\ref{assm:prop} having $f(\theta\mid y,x)$ as its invariant density. Developing constraints on its geometric rate of convergence requires the following assumptions.

\begin{assumption}\label{assm:glm:diff}
	All of $\Pi$, $c$ and $g$ are twice continuously differentiable in $\theta$.
\end{assumption}

\begin{assumption}\label{assm:glm:Pi:bdd}
	If $\| \cdot \|_{sp}$ denotes the spectral norm, there exists  $\lambda(x_1,x_2,\cdots,x_n) \in \mathbb{R}_+$, such that \[\Vert\nabla \Pi(\theta,x_i)\Vert_{sp} \le \lambda(x_1,x_2,\cdots,x_n)<\infty\] for all \(i,\theta\).
\end{assumption}

\begin{assumption}\label{assm:glm:cumulant}
	There exists \(K(x_1,x_2,\cdots,x_n) \in \mathbb{R}_+\) such that either, for all \(i,\theta\), \[\Vert\nabla c(\theta,x_i)\Vert\le K(x_1,x_2,\cdots,x_n)<\infty,\] or 
	\[0\le \nabla^2 c(\theta,x_i) \le K(x_1,x_2,\cdots,x_n) \, I .\]
\end{assumption}

\begin{assumption}\label{assm:prior:Lipschitz}
The gradient of the log-likelihood of the prior is Lipschitz with Lipschitz constant, $\lambda_2$, that is \[\left\Vert\nabla g(\theta_1)-\nabla g(\theta_2)\right\Vert \le \lambda_2 \left\Vert\theta_1-\theta_2\right\Vert \]
\end{assumption}

Finally, one of the following two assumptions will be required.

\begin{assumption}\label{assm:prior:cnvx}
Let $g(x)$ be  $\lambda_1$ strongly convex, that is, there exists $\lambda_1 >0$ such that
$g(\theta_1)-g(\theta_2)\ge \left<\nabla g(\theta_2),(\theta_1-\theta_2)\right> +\frac{\lambda_1}{2}\Vert\theta_1-\theta_2\Vert^2$.
\end{assumption}

\begin{assumption}\label{assm:prior:diss}
Let $g(x)$ be $(a_{\dag},b_{\dag})$-dissipative in which case
there is some $a_{\dag}>0$ and some $b_{\dag}\ge 0$ such that for any $x\in \mathbb{R}^p$,  $x^{\mathsf{T}}\nabla g(x)>a_{\dag}\Vert x\Vert^2-b_{\dag}$.
\end{assumption}

Some key quantities need to be defined before proceeding.  Consider any $\delta$ from  \eqref{delta:defn}. Define $\gamma=\lambda_1$ when $g(\cdot)$ is $\lambda_1$-strongly convex and $\gamma=a_{\dag}$ when $g(\cdot)$ is $(a_{\dag},b_{\dag})-$dissipative. The superexponential decay rate function $f_s(\cdot):\mathbb{R}_+\to \mathbb{R}_+$ is then
$f_s(u)=c \, u.$ Based on the properties of the cumulant function and the prior,  four cases are considered and the expressions for $C_1$ are given below: (i) if $c(\cdot,x)$ has bounded derivative and $g(\cdot)$ is strongly convex, then  
\begin{equation}
\label{eq:C_1_1}
C_1=\lambda(x_1,x_2,\cdots,x_n) \sum_{i=1}^{n}\left|T(y_i)\right|+n\, K(x_1,x_2,\cdots,x_n)+\left|\nabla g(0)\right|;
\end{equation}
(ii) if $c(\cdot,x)$ is convex and $g(\cdot)$ is strongly convex, then
\begin{equation}
\label{eq:C_1_2}
C_1 =\lambda(x_1,x_2,\cdots,x_n) \sum_{i=1}^{n}\left|T(y_i)\right| +n\max_{1\le i\le n} \left|\nabla c(0,x_i)\right|+\left|\nabla g(0)\right|;
\end{equation}
(iii) if $c(\cdot,x)$ has bounded derivative and $g(\cdot)$ is $(a_{\dag},b_{\dag})$-dissipative, then 
\begin{equation}
\label{eq:C_1_3}
C_1 =\lambda(x_1,x_2,\cdots,x_n) \sum_{i=1}^{n}\left|T(y_i)\right|+n\, K(x_1,x_2,\cdots,x_n)+b_{\dag};
\end{equation}
or (iv) if $c(\cdot,x)$ is convex and $g(\cdot)$ is $(a_{\dag},b_{\dag})$-dissipative, then
\begin{equation}
\label{eq:C_1_4}
C_1 =\lambda(x_1,x_2,\cdots,x_n) \sum_{i=1}^{n}\left|T(y_i)\right| +n\max_{1\le i\le n} \left|\nabla c(0,x_i)\right|+b_{\dag}.
\end{equation}
Also define
\begin{equation}\label{keyquantities:K}
    \begin{split}
	    K_1&=\max_{1\le i \le n} \left\Vert\Pi(0,x_i)\right\Vert_{sp}\sum_{i=1}^{n} \left\Vert T(y_i)\right\Vert+n\max_{1\le i \le n}\left\Vert c(0,x_i)\right\Vert+\left\Vert g(0)\right\Vert,\\
	    K_2&=\lambda(x_1,x_2,\cdots,x_n)\sum_{i=1}^{n} \left\Vert T(y_i)\right\Vert+\max_{1\le i \le n}\left\Vert\nabla c(0,x_i)\right\Vert+\left\Vert\nabla g(0)\right\Vert,\\
	    K_3&=\frac{1}{2}\left(\lambda_2+K(x_1,x_2,\cdots,x_n)\right),\\
	    f_{x,y}(u)&=K_1+K_2\, u+K_3\, u^2,\\
    \end{split}
\end{equation}
and 
\begin{equation}\label{keyquantities:radii}
    \begin{split}
	  M'_p&=\max\left[\frac{1}{\gamma-\eta \lambda_2} \left(C_1+\lambda \sum_{i=1}^{n} \left\Vert T(y_i)\right\Vert+\tilde{J}(x_1,x_2,\cdots,x_n)+\left\Vert\nabla g(0)\right\Vert\right),1\right],\\
    \mathcal{R}_{\epsilon}&=\max\left\{\frac{1}{\gamma}\left(C_1-\frac{1}{\delta}\log\epsilon\right)+K_{\epsilon},\ 2K^2_{\epsilon}\left[\left(\frac{\epsilon}{q(0)\, C_B(p) \delta}\right)^{1/p-1} -K_{\epsilon}\right]^{-1}+K_{\epsilon},\right.\\
& \left.\quad \quad \quad \frac{1}{\gamma}\left(C_1-\log \left(\exp\left\{-f_{x,y}\left(\frac{1}{\gamma}\left(C_1\vee 1\right)\right)\right\}/p^*\wedge 1\right)\right)+1, \ M^*+K_{\epsilon}\vee 1\right\}\\
\mathcal{R}_{\max}&=f_s^{-1}\left(C_1-\log\left(\frac{(1-\tilde\lambda)^2}{2\,p^* \,(2b+\epsilon_{\alpha})^2}\wedge 1\right)\right)\vee M^*
    \end{split}
\end{equation}
The constant  $b$ is defined in \eqref{drift:const}. Note that $M^*=M_p=M'_p$ and $\epsilon_{\alpha}=\eta$ for $0<\eta<\gamma/\lambda_2$. Note that here $\gamma=\lambda_1$ where the $g(\cdot)$ is strongly convex and $\gamma=a_{\dag}$ in the case $g(\cdot)$ is $(a_{\dag},b_{\dag})$-dissipative. Also note that $\epsilon$ is as defined in \eqref{eps_def_1}.

\begin{proposition}\label{prop1}
   Let Assumptions~\ref{assm:glm:diff}--\ref{assm:prior:Lipschitz} hold and let either Assumption~\ref{assm:prior:cnvx} or Assumption~\ref{assm:prior:diss} be true. Consider a RWMH algorithm with proposal satisfying assumption~\ref{assm:prop} with stationary distribution as the distribution corresponding to the density given in~\eqref{glmposterior}. Let
 \[\tilde{\eta}=\exp\left\{-2\left(K_1+K_2\mathcal{R}_{\max}+K_3\mathcal{R}_{\max}^2\right)\right\}q(\mathcal{R}_{\max})\, \mu^{Leb}(B(0,\mathcal{R}_{\max})),\] 
 then the RWMH is geometrically ergodic and  	\begin{align*}
	  1 -  \inf_{x \in \mathbb{R}^p} \left[ \int_{\mathbb{R}^p} \alpha(x,y) q(x,y) dy\right]\le\rho\le \max\left\{\left(1-\tilde\eta\right)^{r},\tilde{\alpha}^{-(1-r)}A^r\right\} .
\end{align*}
\end{proposition}
\begin{proof}
    The proof is a direct consequence of Lemma~\ref{glmrepslemma}, Lemma~\ref{glmr1lemma}, and Theorem~\ref{thm:main_result}.
\end{proof}

Note that in Proposition~\ref{prop1} the rate of convergence is bounded by functions of the data.

\subsection{Bayesian Logistic Regression}

Consider Bayesian Logistic Regression with likelihood
\begin{equation}
\label{egdensity}
Y_i \mid X_i, \theta \stackrel{indep}{\sim} Ber\left(\frac{e^{x^{T}_i\theta}}{1+e^{x^{T}_i\theta}}\right)
\end{equation}
and prior measure
\begin{equation}
\label{egprior}
\theta  \sim N(0,I).
\end{equation}
 
The posterior  is characterized by 
\begin{align}
\label{logisticposterior}
f\left(\theta\mid y,x\right)& \propto \exp\left\{\sum_{i=1}^{n}y_ix^{T}_i\theta -\sum_{i=1}^{n}\log\left(1+e^{x^{T}_i\theta}\right)-\frac{1}{2}\Vert\theta\Vert^2 \right\}.
\end{align}
The goal here is to establish upper bounds for the rate of convergence for an RWMH. 
For any $\epsilon$ as defined in \eqref{eps_def_1}, $0<\eta<1/(1+\lambda_{max}(X^{\mathsf{T}}X)/4)$, $C_1$ as in \eqref{eq:C_1_1}-\eqref{eq:C_1_4} with $\lambda(x_1,x_2,\cdots,x_n)=n\, \max_{i} \|x_i\|$, $T(y_i)=y_i$, $K(x_1,x_2,\cdots,x_n)=\lambda(x_1,x_2,\cdots,x_n)$ and \[\delta=\frac{\epsilon}{2\, q(0)\, C_B(p)\, K^{p-1}_{\epsilon}}\] define
\[f_{x,y}(u)=\max_{1\le i\le n}\left\Vert x_i\right\Vert\left(n+\sum_{i=1}^{n}\left\Vert y_i\right\Vert\right)\, u+u^2\]
and 
\begin{equation}
\label{logisticregconstants}
    \begin{split}
    M'_p&=\max\left[\frac{1}{1-\eta\left(1+\lambda_{max}(X^{\mathsf{T}}X)/4\right)}\left(C_1\, + \max_{1\le i\le n}\left\Vert x_i\right\Vert\sum_{i=1}^{n}\left\Vert y_i\right\Vert+n \right),1\right]\\
    \mathcal{R}_{\epsilon}&=\max\left\{C_1-\frac{1}{\delta}\log \epsilon +K_{\epsilon}, 2\, K_{\epsilon}\left[2^{1/(p-1)}-1\right]^{-1}+K_{\epsilon},\right.\\
    &\left.\quad \quad C_1+f_{x,y}(C_1\vee 1)+\log(p^*\wedge 1)+1,M^*+K_{\epsilon}\vee 1\right\}\\
    b&=3\, \exp\left(\frac{1}{2}f_{x,y}(\mathcal{R}_{\epsilon})\right).
    \end{split}
\end{equation}
Note that in this setting,
$f_s(x)=x$ and $M^*=M'_p$ and hence  $\mathcal{R}_{\max}$ may be calculated using \eqref{keyquantities:radii}.

Recall that $A,\tilde{\alpha}$ are defined in \eqref{A:defn} and \eqref{tilde_alpha:defn}, respectively.

\begin{theorem}
\label{logisticthm}
If
\[
\tilde{\eta}=\exp\left\{-2\,\left[\max_{1\le i\le n}\left\Vert x_i\right\Vert\left(n+\sum_{i=1}^{n}\left\Vert y_i\right\Vert\right)\, \mathcal{R}_{\max}+\mathcal{R}_{\max}^2\right]\right\}\, q(\mathcal{R}_{\max}) \mu^{Leb}\left(B(0,\mathcal{R}_{\max})\right),
\]
then the RWMH algorithm with invariant density defined in \eqref{logisticposterior} is geometrically ergodic with 
\begin{align*}
    	  1 -  \inf_{x \in \mathbb{R}^p} \left[ \int_{\mathbb{R}^p} \alpha(x,y) q(x,y) dy\right]\le \rho\le \max\left\{\left(1-\tilde{\eta}\right)^r,\tilde{\alpha}^{-(1-r)}A^r\right\}.
\end{align*}
\end{theorem}
\begin{proof}
    See Appendix~\ref{app:logisticthem}.
\end{proof}

\section{Convergence Lower Bounds Using Spectral Theory}
\label{sec:spectral_results}

Let $\mathbb{L}_2(F)$ be the Hilbert space of all functions that are square integrable with respect to $F$ and having inner product 
\[
\left\langle h,g \right\rangle_{2} = \int h(x) g(x) \, dF(x)
\] 
and associated norm denoted $\Vert \cdot \Vert_2$.  Also $\mathbb{L}_{2,0}(F)$ is the set of all functions in $\mathbb{L}_2(F)$ with $\int h\, dF=0$.  Recall that $P$, defined in \eqref{RWMH:kernel}, is the Markov kernel for the RWMH, which is reversible with respect to $F$ by construction.  For any $h \in \mathbb{L}_{2,0}(F)$, 
\[
Pf(x) := \int f(y) P(x, dy) \in \mathbb{L}_{2,0}(F)
\] 
and hence $P$ may be regarded as a positive linear operator on $\mathbb{L}_{2,0}(F)$ \citep{10.1214/ECP.v18-2507}. The spectrum of $P$ is the set 
\[
S(P)=\left\{\lambda \in \mathbb{C}: P-\lambda Id \ \text{is not invertible}\right\}
\]
and $S_0(P)$ is the spectrum under the restriction of the function class being $\mathbb{L}_{2,0}(F)$. The spectral gap of $P$ is
\[
\gamma_P=1-\sup \left|S_0(P)\right|
\]
and define $Gap_R(P)=1-\sup S_0(P)$.

If $h \in \mathbb{L}_2(F)$ and the  Dirichlet form for $P$ is  
\[
\mathcal{E}(P,h)=\left\langle h,h\right\rangle_{2}-\left\langle Ph,h\right\rangle_{2},
\]
then define
\[
Gap_R(P)=\inf_{h\in \mathbb{L}_{2,0},\, h\ne 0}\frac{\mathcal{E}(P,h)}{\|h\|^2_2}.
\]
Moreover, it can be shown that $\gamma_P=Gap_R(P)$.

There are well-known bounds for $Gap_R(P)$ in terms of the conductance of $P$, $\Phi^*_P=\Phi_P(1/2)$ \cite[Theorem 3.5]{lawler1988bounds} with
\begin{align*}
\Phi_P(v)=\inf \left\{\frac{\int_A P(x,A^c) F(dx)}{F(A)}: \, F(A)\le v\right\}, \quad v \in (0,1/2].
\end{align*}
Specifically, a fundamental result \cite[Theorem 3.5]{lawler1988bounds} yields
\[
 \frac{1}{2}\left[\Phi^*_P\right]^2\le Gap_R(P)\le \Phi^*_P.
 \]

The following assumptions are required for the results in this section.
\begin{assumption}
\label{spectral:strngcnvx}
For the measure $F(\cdot)$, $-\log \frac{dF}{d\mu^{Leb}_p}(x)=-\log f(x)=U(x)$ is a m-strongly convex function. Also, for any $x,y \in \mathbb{R}^p$, and some  $L \in \mathbb{R}_+$.
\[
 \left\|\nabla U(x)-\nabla U(y)\right\|\le L\left\|x-y\right\| .
 \]
\end{assumption}

\begin{assumption}\label{spectral:proposalassm0}
	$q(x),\, x \in \mathbb{R}$ has mode $0$ with 
	$-\log q(x)=V_{\dag}(x)$. Also either $V_{\dag}(x)$ is $m_1$-strongly convex with $m_1>0$ or $V_{\dag}''(0)>0$ with $V_{\dag}'''(\cdot)\ge 0$.
\end{assumption}

\begin{remark}
	Note that this assumption contains some abuse of notation and extension of the domain of a definition. That is, $q(x,y)=q(|x-y|)$ is the proposal density and the domain of definition of the function is $\{0\}\cup \mathbb{R}_+$. We extend this by defining $q(x)=q(-x)$.
\end{remark}

\begin{assumption}\label{spectral:proposalassm}
	$q(x),\, x \in \mathbb{R}$ has mode $0$ with 
	$-\log q(x)=V_{\dag}(x)$. Also,  assumption~\ref{spectral:strngcnvx} is true, then $V_{\dag}''(x)\ge m_1>-m$ for all $x \in \mathbb{R}$. Finally,
	\[\frac{\left(2\pi\right)^{p/2} e^{-V_{\dag}(0)}}{\left(m+m_1\right)^{p/2}}<1.\]
\end{assumption}

\begin{proposition}\label{Spectral:lwr:bnd1}
    Under assumptions~\ref{spectral:strngcnvx} and ~\ref{spectral:proposalassm0}, the RWMH algorithm with $F(\cdot)$ as the stationary distribution and $Q(\cdot,\cdot)$ as the proposal is $(M,\rho)$-geometrically ergodic, with
    \[\rho \ge 1-\frac{1}{2}\,L\,\frac{p\,e^{-V_{\dag}(0)}} {\tilde{J}^{p/2+1}}\]
\end{proposition}
\begin{proof}[Proof of Proposition~\ref{Spectral:lwr:bnd1}]
      The fact that the RWMH is geometrically ergodic follows from Theorem~\ref{expfam:ergdocitythm} with $x=0$ and $c(\theta,0)=0$. Therefore there exists a $\rho$ such that 
    \[\|P^n(x,\cdot)-F(\cdot)\|_{TV}\le M(x) \rho^n. \]
    From previous work~\citep[Theorem 3]{roberts_tweedie_2001}, we know that 
    \[\|P^n(x,\cdot)-F(\cdot)\|_{2}\le C(x) \rho^n\]
    for some function $C(x)>0$. We also know that $\rho=1-Gap_R(P)$. Hence using Lemma~\ref{prop:spectral:gap}, the result follows.
\end{proof}
\begin{proposition}\label{Spectral:lwr:bnd2}
	Under assumptions~\ref{spectral:strngcnvx} and~\ref{spectral:proposalassm}, the RWMH algorithm with $F(\cdot)$ as the stationary distribution and $Q(\cdot,\cdot)$ as the proposal is $(M,\rho)$-geometrically ergodic, with
    \[\rho \ge 1-\frac{\left(2\pi\right)^{p/2} e^{-V_{\dag}(0)}}{\left(m+m_1\right)^{p/2}}\]
\end{proposition}
\begin{proof}[Proof of Proposition~\ref{Spectral:lwr:bnd2}]
  The proof of Proposition~\ref{Spectral:lwr:bnd2} is identical to that of Proposition~\ref{Spectral:lwr:bnd1} with Lemma~\ref{prop:spectral:gap} being replaced by Lemma~\ref{prop:conductance} .
\end{proof}

\section{Final Remarks}
In addition to improving the sufficient conditions for geometric ergodicity, the first explicit constraints on the geometric rate of convergence for RWMH on $\mathbb{R}^{p}$ with a general proposal were developed.  These constraints allow one to begin to study how various features might impact the geometric convergence rate of RWMH. However, the computable upper bounds were obtained by appealing to well known results~\citep{rose:1995a}, which even though they are often conservative can produce practical, reasonable bounds  \citep{jone:hobe:2001, jone:hobe:2004, rose:1996}. 
 Moreover, the quality of the upper bounds depends on the quality of the drift and minorization conditions.  The approach taken here used the same drift function as previous work. Alternative drift functions for specific target distributions or improved inequalities could result in sharper upper bounds.    

The approach to establishing explicit drift and minorization conditions for RWMH appears to have potential for wider application to other MH algorithms.  For example, recent work established sufficient conditions for geometric ergodicity of Metropolis Adjusted Langevin Algorithms ~\cite{roy2021convergence} and the approach taken here would seem to apply in developing explicit constraints on its geometric rate of convergence.
 
\bibliographystyle{apalike}
\bibliography{mcref}

\begin{appendices}\label{appendix}
 
\section{Proofs for Section~\ref{sec:GE_RWMH}}
 
\subsection{Theorem~\ref{thm:rwmh_ge}}
\label{app:GE_RWMH}
Establishing Theorem~\ref{thm:rwmh_ge} concerns exploiting assumption~\ref{assm:cone} to obtain a bound independent of the current location of the algorithm.. 
Consider the Lyapunov function $V(x)=f(x)^{-1/2}$ and let $R(x)=A(x)^c$.  By assumption $f$ is superexponentially light and hence \citep[][p. 352]{jarn:hans:2000} 
\begin{align*}
\lim_{\|x\|\to \infty} \frac{PV(x)}{V(x)} = \lim_{\|x\|\to \infty} Q(x,R(x))
\end{align*}
Note that \(Q(x,R(x))=1-Q(x,A(x))\). Using assumption~\ref{assm:cone}, there exists an $M_p$ such that for all $x\in \mathbb{R}^p$ with $\|x\|>M_p$, $W(x) \subset A(x)$. Hence for any $x$ with $\|x\|>M_p$, $Q(x,A(x))\ge Q(x,W(x))$ where $Q(x,W(x))$ does not depend on $x$ by the assumption on the proposal kernel. Then  
\[
\lim_{\|x\|\to \infty} \frac{PV(x)}{V(x)} \le  \lim_{\|x\|\to \infty} 1-Q(x,W(x)) =1-Q(0,W(0)).
\]
Hence the drift condition is satisfied. The minorization follows since the proposal is unbounded off of compact sets and hence the proposal can be minorized using Lebesgue measure \citep[][Theorem 2.2]{robe:twee:1996}.

\subsection{Proposition~\ref{prop:key}}
\label{app:key}
We shall consider the two cases separately. In the first case, we exhibit a cone in the acceptance region given the curvature condition. The second case follows from two lemmas which are subsequently established.
\begin{enumerate}
  \item[(i)] Let $n(x)=x/\Vert x\Vert$ and $m(x)=\nabla f(x)/\Vert\nabla f(x)\Vert$.  The curvature condition~\eqref{eq:curvature} implies there exists $M_p$ such that \(n(x)\cdot m(x)\le -\eta\) for some $1\ge \eta>0$ and $\Vert x\Vert>M_p$.  Consider vectors $\xi \in \mathbb{S}^{p-1}$ such that $\Vert \xi -n(x)\Vert\le \eta/2$. Let the angle between $\xi$ and $n(x)$ be denoted by $\beta_{x,\xi}$.
  one can see that 
  \begin{align*}
      \left\Vert\xi -n(x)\right\Vert^2\le \eta^2/4
  \end{align*}
  which in turn implies 
  \[\Vert\xi\Vert^2+\Vert n(x)\Vert^2-2\Vert n(x)\Vert\Vert\xi\Vert\cos\beta_{\xi,x}\le \eta^2/4.\]
  This implies \[2-2\cos\beta_{x,\xi} \le \eta^2/4,\] which in turn implies,  \[0\le \beta_{x,\xi}\le \cos^{-1}\left(1-\eta^2/8\right)\]
	where $\cos^{-1}$ to be between $[0,\pi/2]$. Also note that there is an $M_p$ such that any $K>0$ can be accommodated, that is, choose new $M_p$ as $M_p+K$ and the results still hold. 

\item[(ii)]
The proof immediately follows from the following two Lemmas.  

\begin{lemma}\label{coneconvexlemma}
	If $\log f(x)$ is a concave function, then $A(x)$ is a convex set with nonempty interior and hence admits a cone of some radius $r>0$ and angle $\alpha$ at all boundary points
\end{lemma}
 \begin{proof}
 	Note that the region of sure acceptance is given by \[A(x)=\left\{y: f(y)\ge f(x)\right\}.\]
 	Note that this is equivalent to \[A(x)=\left\{y: \log f(y)\ge \log f(x)\right\}\]
 	since $f>0$ in $\mathbb{R}^p$. Consider any $0\le \kappa \le 1$ and note that for $y_1,y_2 \in A(x)$,
 	\[\log f(\kappa y_1+(1-\kappa)y_2)\ge \kappa \log f(y_1)+(1-\kappa)\log f(y_2)\ge \log f(x).\]
 	This implies $\kappa y_1+(-\kappa)y_2 \in A(x)$. Therefore $A(x)$ is a convex set with non-empty interior. Hence every boundary point of $A(x)$ admits a cone of some radius and some angle. 
 \end{proof} 
 
\begin{lemma}\label{lemmingway}
    If $f$ is superexponential and $A(x)$ is a convex set with nonempty interior for every $x$ with $\Vert x\Vert >\tilde{M}_p$, then assumption~\ref{assm:cone} is satisfied for some $M_p$.
\end{lemma}
\begin{proof}
    Note that since $f$ is superexponentially light there exists an $R_{f}\in \mathbb{R}$ such that $0 \in A(x)\subset A(y)$ whenever $\Vert y\Vert >\Vert x\Vert >R_{f}$. Consider for each $z \in \partial A(x)$, the set  \[W^K_{\epsilon}(z)=\left\{y=z-a\xi : 0<a<K, \ \xi\in \mathbb{S}^{p-1}, \ \left\Vert\xi-n(z)\right\Vert\le \epsilon/2 \right\}.\]
    Note that by the two facts of $A(x)$ being convex and $f$ being continuous on $\mathbb{R}^p$, $W^K_{\epsilon}(z)\subset A(x)$. This is easy to see constructively by considering any point in $b_{int} \in A(x)^o$ (the interior of $A(x)$). Note that there exists an $r>0$ such that $B(b_{int},r)\subset A(x)$ (where $B(b_{int},r)$ is the $p$-dimensional open ball centered at $b_{int}$ with radius $r>0$) as $f$ is continuous and hence $A(x)^o$ is an open set. Join all points from $z$ to $B(b_{int},r)$ with $z,b_{int},0$ colinear. Note that such a $z$ always exists. Also note that by convexity this constructed set is in $A(x)$. It is easy to see that for some $\epsilon$ and $K$ this set contains $W^K_{\epsilon}(z)$. Next we prove that one always finds such a cone even if $\Vert x\Vert$ increases.
      Consider the case where the statement of the Lemma fails for each $z$ on the boundary of $A(x)$ i.e. $\epsilon \to 0$ as $\Vert x\Vert\to \infty$. By our hypothesis and \citep{jarn:hans:2000} there shall exist three points at the boundary points $\tilde{z}_1, \tilde{z}_2$ and $\tilde{z}_3$ such that each shall be the vertex of a cone $W^K_{\epsilon}(\tilde{z}_i)$ where $\epsilon$ in this case is the maximum permissible angle. This immediately implies that for some $\lambda$ there exists points $z_1,z_2$ such that $\lambda z_1+(1-\lambda)z_2 \not\in A(x)$. This is true as if it is not, we can indeed find a cone of larger angle than $\epsilon$ by joining all points of $W^K_{\epsilon}(\tilde{z}_i)$ by straight lines. Hence $A(x)$ is not convex and we have a contradiction. The only case that remains is when $A(x)$ converges to a lower dimensional subspace as $\Vert x\Vert\to \infty$. This is however again not true as $f$ is superexponentially light and hence there exist a sequence of $x_n$ such that $A(x_n) \subset A(x_{n+1})$. 
   \end{proof}
   \end{enumerate}

\subsection{Proof of Proposition~\ref{prop:curvature1}}
\label{app:curvature1}
The proof follows using by simple manipulations.
Note that 
	\begin{align*}
	\left\langle \frac{x}{\Vert x\Vert}, \frac{\nabla f(x)}{\Vert \nabla f(x)\Vert}\right\rangle &=	\left\langle \frac{x}{\Vert x\Vert}, \frac{\nabla \log f(x)}{\Vert \nabla \log f(x)\Vert}\right\rangle\\ 
	&\le \frac{C_1}{\Vert \nabla \log f(x)\Vert}-\frac{f_s(\vert x\vert)}{\left\Vert \nabla \log f(x)\right\Vert}\\
	&\le -\eta.
	\end{align*}
	This implies that 
	\[\limsup_{\|x\|\to\infty} \left\langle \frac{x}{\Vert x\Vert}, \frac{\nabla f(x)}{\Vert \nabla f(x)\Vert}\right\rangle \le -\eta.\]
	Hence, using Proposition~\ref{prop:key}, the proof is immediate.

\subsection{Proof of Theorem~\ref{expfam:ergdocitythm}}
\label{app:proof of expfam:ergdocitythm}
The following calculation will establish that the density in question is both superexponentially light and log-concave and hence the result will follow from Proposition~\ref{prop:key} and Theorem~\ref{thm:rwmh_ge}.
Note that 
    \begin{align*}
        \left\langle\frac{\theta}{\left\Vert\theta\right\Vert}, \nabla \log f(\theta\mid X_1,X_2,\cdots,X_n)\right\rangle &=        \left\langle\frac{\theta}{\left\Vert\theta\right\Vert}, \sum_{i=1}^{n}x_i-n\,\nabla c(\theta) -\nabla g(\theta)\right\rangle\\
        &=\frac{1}{\left\Vert\theta\right\Vert}\left(\sum_{i=1}^{n}\theta^{\mathsf{T}}x_i-n\,\theta^{\mathsf{T}}\nabla c(\theta)-\theta^{\mathsf{T}}\nabla g(\theta)\right)\\
        &=\frac{1}{\left\Vert\theta\right\Vert}\Bigg(\sum_{i=1}^{n}\theta^{\mathsf{T}}x_i-n\,\theta^{\mathsf{T}}\nabla c(0)- n\, \theta^{\mathsf{T}}\nabla^2 c(\xi_{1,\theta})\theta\\
        &\quad \quad -\theta^{\mathsf{T}}\nabla g(0)- \theta^{\mathsf{T}}\nabla^2 g(\xi_{2,\theta})\theta\Bigg)
    \end{align*}
    where $\xi_{1,\theta}, \xi_{2,\theta}$ are values in the line joining $0$ and $\theta$.
    Note that since $c(\theta)$ is convex \citep[Theorem 7.1]{barndorff2014information}, $\theta^{\mathsf{T}}\nabla^2 c(\xi_{1,\theta})\theta\ge 0$. Also since $g(\cdot)$ is strongly convex,  $\theta^{\mathsf{T}}\nabla^2 g(\xi_{2,\theta})\theta\ge m \, |\theta|^2$. Therefore
    \begin{align*}
     \left\langle\frac{\theta}{\left\Vert\theta\right\Vert}, \nabla \log f(\theta\mid X_1,X_2,\cdots,X_n)\right\rangle &\le    \frac{1}{\left\Vert\theta\right\Vert}\left(\sum_{i=1}^{n}\theta^{\mathsf{T}}x_i-n\,\theta^{\mathsf{T}}\nabla c(0)- \theta^{\mathsf{T}}\nabla g(0)- m\,\left\Vert\theta\right\Vert^2\right)\\
     &=\underset{I}{\underline{\frac{1}{\left\Vert\theta\right\Vert}\left(\sum_{i=1}^{n}\theta^{\mathsf{T}}x_i-n\,\theta^{\mathsf{T}}\nabla c(0)- \theta^{\mathsf{T}}\nabla g(0)\right)}}-\underset{II}{\underline{m\,\left\Vert\theta\right\Vert}}
    \end{align*}
    Note that 
    \begin{align*}
        \left\Vert I\right\Vert \le \sum_{i=1}^{n}\left\Vert x_i\right\Vert +n\left\Vert\nabla c(0)\right\Vert+\left\Vert\nabla g(0)\right\Vert
    \end{align*}
    and 
    $II \to \infty$ as $\Vert \theta\Vert\to \infty$.
    
    Thus,
    \[\limsup_{\Vert\theta\Vert\to \infty}    \left\langle\frac{\theta}{\left\Vert\theta\right\Vert}, \nabla \log f(\theta\mid X_1,X_2,\cdots,X_n)\right\rangle =-\infty.\]
    Hence the posterior is superexponentially light. Also it is easy to see 
    \[-\nabla^2 \log f(\theta\mid X_1,X_2,\cdots,X_n)=n\, \nabla^2 c(\theta)+\nabla^2 g(\theta) >0\]
    and hence the density is log-concave. Therefore the proof follows from Proposition~\ref{prop:key} and Theorem~\ref{thm:rwmh_ge}.

\section{Proofs for Section~\ref{sec:main_results}}

\subsection{Drift and Minorization for RWMH}
\label{sec:dm}

A preliminary result is required before turning to the proofs of Propositions~\ref{prop:driftJH} and~\ref{prop:minorJH}.
\begin{lemma}
\label{volumelemma}
Let Assumption~\ref{assm:cone} hold, with $K\le 1/3$ and $\epsilon_{\alpha}<\tilde{\epsilon}_{\alpha}\wedge 1/3$. Then for  
$$R_{\alpha}=K\left(1-\epsilon_{\alpha}^2/8\right)\left[1-\left(1-\epsilon_{\alpha}^2/8\right)^2\right]^{1/2}\left\{1+\left[1-\left(1-\epsilon_{\alpha}^2/8\right)^2\right]^{1/2}\right\}^{-1}$$
and $|x|>M_p$,
\begin{align*}
Q(x,A(x))\ge Q\left(0,\left[-\frac{R_{\alpha}}{\sqrt{p}}-1,\frac{R_{\alpha}}{\sqrt{p}}-1\right]\times \left[-\frac{R_{\alpha}}{\sqrt{p}},\frac{R_{\alpha}}{\sqrt{p}}\right]^{p-1}\right)>0.
\end{align*}
\end{lemma}

\begin{proof}[Proof of Lemma~\ref{volumelemma}]
The main idea of the proof is to find a geometric structure that is invariant as $A(x)$ increases; the volume of which can be easily calculated.
		Let $W^K_{\epsilon_{\alpha}}(x)$ be as in \eqref{eq:cone}. Then by assumption~\ref{assm:cone}, $W^K_{\epsilon_{\alpha}}(x)\subset A(x)$ when $|x|>M_p$ and hence \(Q(x,W^K_{\epsilon_{\alpha}}(x))\le Q(x,A(x))\). Note that the $Q$ volume of $W^K_{\epsilon_{\alpha}}(x)$ does not depend on $x$ due to the structure of $W^K_{\epsilon_{\alpha}}(x)$ and $Q$ being radially invariant.
		
		 Without loss of generality, choose $x=(1,0,0,\cdots,0)^{\mathsf{T}}$. 
		 Observe that $\mathbb{S}^{p-1}$ or $\mathbb{D}^p$ which are respectively the hollow sphere centred at the origin with radius $1$ and the closed ball with radius $1$ both pass through $x$. 
 Consider the set of $(\xi_1,\xi_2,\cdots,\xi_p)$ such that
\[
\left\{ 1-K\left(1-\epsilon_{\alpha}^2/8\right)<\xi_1<1, \ 
		  \xi_2^2+\xi_3^2+\cdots+\xi_p^2
		 \le \left(1-\xi_1\right)^2\left(1-\left(1-\epsilon_{\alpha}^2/8\right)^2\right)\left(1-\epsilon_{\alpha}^2/8\right)^{-1}\right\}.
\]
Call this set $\mathcal{C}$.
Note that $\mathcal{C} \subset W(x)$. Define 
		\[
		R_{\alpha}=K\left(1-\epsilon_{\alpha}^2/8\right)\left[1-\left(1-\epsilon_{\alpha}^2/8\right)^2\right]^{1/2}\left\{1+\left[1-\left(1-\epsilon_{\alpha}^2/8\right)^2\right]^{1/2}\right\}^{-1}.
		\]
		Define the ball centered at $1-K(1-\epsilon_{\alpha}^2/8)+R_{\alpha}$ with radius $R_{\alpha}$ as \(B(1-K(1-\epsilon_{\alpha}^2/8)+R_{\alpha},R_{\alpha})\). Consider the closed sphere $R_{\alpha}\cdot \mathbb{D}^{p}$ with radius $R_{\alpha}$ centered at the origin. By assumption~\ref{assm:cone}, 
		\[
		Q(x,R_{\alpha}\cdot \mathbb{D}^{p}) \le Q(x,B(1-K(1-\epsilon_{\alpha}^2/8)+R_{\alpha},R_{\alpha}) ).
		\] 
		Hence, 
		\[
		Q(x,R_{\alpha}\cdot \mathbb{D}^{p}) \le Q(x, B(1-K(1-\epsilon_{\alpha}^2/8)+R_{\alpha},R_{\alpha}))\le Q(x,C) \le Q(x,W(x)).
		\] 
		Consider a $p$-dimensional cube 
		\[
		\mathbb{A}=\left[-\frac{R_{\alpha}}{\sqrt{p}},\frac{R_{\alpha}}{\sqrt{p}}\right]^p
		\]
		so that $\mathbb{A} \subset R_{\alpha}\cdot \mathbb{D}^{p}$. This implies
		\[
		Q(x,\mathbb{A})\le Q(x,R_{\alpha}\cdot \mathbb{D}^{p}) \le Q(x, C)\le Q(x,W(x)).
		\]
		Now, from the definition of the proposal, one has
		\begin{align*}
		Q(x,\mathbb{A}) &=  \int_{-\frac{R_{\alpha}}{\sqrt{p}}}^{\frac{R_{\alpha}}{\sqrt{p}}}\int_{-\frac{R_{\alpha}}{\sqrt{p}}}^{\frac{R_{\alpha}}{\sqrt{p}}}\cdots \int_{-\frac{R_{\alpha}}{\sqrt{p}}}^{\frac{R_{\alpha}}{\sqrt{p}}} q\left(\left(\left(\xi_1-1\right)^2+\xi^2_2+\cdots\xi^2_p\right)^{1/2}\right) d\xi_1 d\xi_2 \cdots d\xi_p\\
		&= \int_{-\frac{R_{\alpha}}{\sqrt{p}}-1}^{\frac{R_{\alpha}}{\sqrt{p}}-1}\int_{-\frac{R_{\alpha}}{\sqrt{p}}}^{\frac{R_{\alpha}}{\sqrt{p}}}\cdots \int_{-\frac{R_{\alpha}}{\sqrt{p}}}^{\frac{R_{\alpha}}{\sqrt{p}}} q\left(\left(\xi_1^2+\xi^2_2+\cdots\xi^2_p\right)^{1/2}\right) d\xi_1 d\xi_2 \cdots d\xi_p\\
		&=Q\left(0,\left[-\frac{R_{\alpha}}{\sqrt{p}}-1,\frac{R_{\alpha}}{\sqrt{p}}-1\right]\times \left[-\frac{R_{\alpha}}{\sqrt{p}},\frac{R_{\alpha}}{\sqrt{p}}\right]^{p-1}\right).
		\end{align*}
		By assumption \ref{assm:cone} this quantity is greater than zero as the Lebesgue measure of this set is positive. Since \(Q(x,A(x))\ge Q(x,W(x)) \) we are done.
	\end{proof}
\begin{remark}
    Note that taking $K\le 1/3$ does not effect the generality of the proof as if $K>1/3$, we can just select $K=1/3$ and then we are done.
\end{remark}
\begin{proposition}
\label{prop:driftJH}
Let Assumptions~\ref{assm:super}, \ref{assm:cone}, \ref{assm:envelope} and ~\ref{assm:prop} hold. Then, for any $\epsilon$ as defined in \eqref{eps_def_1}, there exists $\mathcal{R}_{\epsilon} > M_p$ such that if $V(x)=\left[f(x)\right]^{-1/2}$ and $b=3\exp\{\tilde{f}(\mathcal{R}_{\epsilon})/2\}$, then
	\begin{align*}
	PV(x)\le \tilde{\lambda}V(x)+b .
	\end{align*}
\end{proposition}

\begin{proof}
The main idea is to exploit a previous result in obtaining key quantities using which we can establish asymptotic bounds quantitatively.
From previous work \cite{jarn:hans:2000}, 
\[
\limsup_{\|x\|\to \infty} \frac{PV(x)}{V(x)} \le \limsup_{\|x\|\to \infty}\left( 1-Q(x,A(x))\right).
\]
Thus for any $\epsilon>0$, there exists $\mathcal{R}_{\epsilon}$ such that if $\|x\| > \mathcal{R}_{\epsilon}$, then
\[
\frac{PV(x)}{V(x)}\le  \limsup_{\|x\|\to \infty} \left(1-Q(x,A(x))\right)+\epsilon.
\]
Notice that we can and will take $\mathcal{R}_{\epsilon} \ge M_p$.

On the set $\|x\|>\mathcal{R}_{\epsilon}$ using Lemma~\ref{volumelemma} in the last inequality, obtain
\begin{align*}
\frac{PV(x)}{V(x)} &\le \limsup_{\|x\|\to \infty} \left(1-Q(x,A(x))\right)+\epsilon\\
&\le \sup_{\|x\|>\mathcal{R}_{\epsilon}} \left(1-Q(x,A(x))\right)+\epsilon\\
&\le \sup_{\|x\|>M_p} \left(1-Q(x,A(x))\right)+\epsilon \\
& \le 1- Q\left(0,\left[-\frac{R_{\alpha}}{\sqrt{p}}-1,\frac{R_{\alpha}}{\sqrt{p}}-1\right]\times \left[-\frac{R_{\alpha}}{\sqrt{p}},\frac{R_{\alpha}}{\sqrt{p}}\right]^{p-1}\right) + \epsilon
\end{align*}

Now choose 
\[
\epsilon=\frac{\epsilon_{\alpha}}{1+\epsilon_{\alpha}}\,Q\left(0,\left[-\frac{R_{\alpha}}{\sqrt{p}}-1,\frac{R_{\alpha}}{\sqrt{p}}-1\right]\times \left[-\frac{R_{\alpha}}{\sqrt{p}},\frac{R_{\alpha}}{\sqrt{p}}\right]^{p-1}\right)\]
and 
\[
\tilde{\lambda} = 1 - \frac{1}{1 + \epsilon_{\alpha}} Q\left(0, \left[-\frac{R_{\alpha}}{\sqrt{p}}-1, \frac{R_{\alpha}}{\sqrt{p}}-1 \right] \times \left[-\frac{R_{\alpha}}{\sqrt{p}},\frac{R_{\alpha}}{\sqrt{p}}\right]^{p-1}\right)
\]
so that for $\|x\|>\mathcal{R}_{\epsilon}$,
\[
PV(x) \le \tilde{\lambda} V(x).
\]

Consider $\|x\| \le \mathcal{R}_{\epsilon}$. We know that \citep[][p. 348]{jarn:hans:2000}  by defining
\[
b  =\sup_{x \in B(0,\mathcal{R}_{\epsilon})} V(x) \cdot \sup \frac{PV(x)}{V(x)},
\]
 on $\|x\| \le \mathcal{R}_{\epsilon}$, 
\[
PV(x) \le b .
\]
Hence, for all $x$,
\[
PV(x) \le \tilde{\lambda}V(x)+b.
\]

By assumption~\ref{assm:envelope}
\begin{align*}
		\exp\left(-\frac{1}{2}\tilde{f}(\Vert x \Vert)\right) \le V(x) \le \exp\left(\frac{1}{2} \tilde{f}(\Vert x \Vert)\right).
\end{align*}
and hence
\[
\sup_{x \in B(0,\mathcal{R}_{\epsilon})} V(x) \le \exp\left(\frac{1}{2} \tilde{f}(\mathcal{R}_{\epsilon})\right)
\]

Note that 
\begin{align*}
    \frac{PV(x)}{V(x)}=\int_{A(x)} \left(\frac{f(x)}{f(y)}\right)^{1/2} q(x,y) dy +\int_{R(x)} \left[1+\left(\frac{f(y)}{f(x)}\right)^{1/2}-\frac{f(y)}{f(x)}\right] q(x,y)dy.
\end{align*}
Hence $PV(x)/V(x)\le 3$ and as a consequence,
\[
b \le 3\exp\left(\frac{1}{2} \tilde{f}(\mathcal{R}_{\epsilon})\right),
\]
which completes the proof.
\end{proof}

\begin{proposition}
\label{prop:minorJH}
If Assumptions~\ref{assm:super}, \ref{assm:cone}, \ref{assm:envelope} and ~\ref{assm:prop} hold, there exists $\mathcal{R}_{\max}>0$ such that if  $V(x)=\left[f(x)\right]^{-1/2}$ and
	\[
	\tilde{\eta}=\exp\left\{-2 \,\tilde{f}(\mathcal{R}_{\max})\right\} q(\mathcal{R}_{\max}) \,\mu^{Leb}(B(0,\mathcal{R}_{\max})),
	\]
	then, there is a probability measure $\nu$ such that
	\begin{align*}
	P(x,\cdot)\ge \tilde{\eta} \,\nu(\cdot) \quad \quad ~x \in \{ V(x) \le (2b+\epsilon_{\alpha})/(1-\tilde{\lambda})\}.
	\end{align*}
\end{proposition}

\begin{proof}
The current result relies on establishing a minorization on a ball with a radius large enough to contain the set \[\left\{ V(x) \le\frac{ (2b+\epsilon_{\alpha})}{(1-\tilde{\lambda})}\right\}.\]
If $d>2b / (1-\tilde{\lambda})$, we establish a  minorization on $\left\{V(x)<d\right\}$. Note that for any such $d$, $B(0,\mathcal{R}_{\epsilon}) \subset \left\{V(x)<d\right\}$. Choose $d=(2b+\epsilon_{\alpha})/(1-\tilde{\lambda})$ and note that
\[
\left\{V(x)<\frac{2b+\epsilon_{\alpha}}{1-\tilde{\lambda}}\right\}=\left\{\frac{\left(1-\tilde{\lambda}\right)^2}{\left(2b+\epsilon_{\alpha}\right)^2}<f(x)\right\}.
\]
By assumption $f(x)$ is superexponentially light so the boundary of this set has a polar representation where radial function is continuous \cite{jarn:hans:2000}. Hence there exists $\mathcal{R}_{\max}$ such that 
\[
\left\{\frac{\left(1-\tilde{\lambda}\right)^2}{\left(2b+\epsilon_{\alpha}\right)^2}<f(x)\right\}\subset B(0,\mathcal{R}_{\max}).
\] 

Indeed, one can choose a $0<\delta<(1-\tilde{\lambda})^2 / (2b+\epsilon_{\alpha})^2$ such that
\[\left\{\frac{\left(1-\tilde{\lambda}\right)^2}{\left(2b+\epsilon_{\alpha}\right)^2}<f(x)\right\} \subset \left\{\delta \le f(x)\right\} .\]  

Note that $B(0,\mathcal{R}_{\epsilon}) \subset B(0,\mathcal{R}_{\max})$. Let  $C=\overline{B(0,\mathcal{R}_{\max})}$ be the closed ball of radius $\mathcal{R}_{\max}$ with center at $0$. Take any $A\subset C$. Take any $x\in C$.
Note that 
\begin{align*}
    P(x,A) &\ge \int_A q(x,y) \left\{\frac{f(y)}{f(x)}\wedge 1\right\} dy\\
    &\ge \frac{\inf_{x \in C} f(x) }{\sup_{x \in C} f(x) } \int_A q(x,y) dy\\
    &\ge   \frac{\inf_{x \in C} f(x) }{\sup_{x \in C} f(x) } \inf_{x,y \in C} q(x,y) \int_A dy\\
    &\ge \frac{\inf_{x \in C} f(x) }{\sup_{x \in C} f(x) } \inf_{x,y \in C} q(x,y) \mu^{Leb}(A)\\
    &= \frac{\inf_{x \in C} f(x) }{\sup_{x \in C} f(x) }  \inf_{x,y \in C} q(x,y) \,\mu^{Leb}(C) \,\frac{\mu^{Leb}(A)}{\mu^{Leb}(C)}.
\end{align*}
Now, using assumption~\ref{assm:envelope},
\begin{align*}
\inf_{x \in C} f(x) \ge \exp\left(-\tilde{f}(\mathcal{R}_{\max})\right)
\end{align*}
and,
\begin{align*}
\sup_{x \in C} f(x) \le \exp\left( \tilde{f}(\mathcal{R}_{\max})\right)
\end{align*}
This implies 
\begin{align*}
P(x,A)\ge \exp\left\{-2 \, \tilde{f}(\mathcal{R}_{\max})\right\} \inf_{x,y \in C} q(x,y) \,\mu^{Leb}(C) \,\frac{\mu^{Leb}(A)}{\mu^{Leb}(C)}.
\end{align*}
Note that by assumption~\ref{assm:cone}, $\inf_{x,y \in C} q(x,y) \ge q(\mathcal{R}_{\max})$ since $q$ is radially decreasing.\\
Define the measure $\nu(\cdot)$ as 
\begin{align*}
    \nu(B)&=\frac{\mu^{Leb}(B\cap C)}{\mu^{Leb}(C)} 
\end{align*}
for any $B \in \mathcal{B}$, the Borel sigma field in question.\\
So, for any $A$, measurable, 
\begin{align*}
    P(x,A) &= P(x,A\cap C)+P(x,A\cap C^c)\\
    &\ge \exp\left\{-2 \,\tilde{f}(\mathcal{R}_{\max})\right\} q(\mathcal{R}_{\max}) \,\mu^{Leb}(C) \,\nu (A)
\end{align*}
Hence considering $\tilde{\eta}=\exp\left\{-2 \,\tilde{f}(\mathcal{R}_{\max})\right\} q(\mathcal{R}_{\max}) \,\mu^{Leb}(C)$ suffices.
\end{proof}

\subsection{Preliminary results for Section~\ref{sec:main_results}}

Next we present important Lemmas critical to obtain the rates for geometric ergodicity under assumptions~\ref{assm:cone}-\ref{assm:super'}.
\begin{lemma}\label{lemmadecay1}
 Under assumption~\ref{assm:super'}, 
 \begin{align*}
     \frac{f\left(x+t\frac{x}{\|x\|}\right)}{f(x)} < \epsilon 
 \end{align*}
 for any $t>0$ if \(\Vert x\Vert >\max (f_s^{-1}(C_1-\log \epsilon/t),M_s)\) where $M_s$ is as defined in assumption~\ref{assm:super'}.
\end{lemma} 
\begin{proof}
 Denote by $n(x)$ the quantity $x/\Vert x\Vert$.
  We have 
  \begin{align*}
  \frac{f(x+t\, n(x))}{f(x)}&=\exp\left\{\log f(x+t\, n(x))-\log f(x)\right\}\\
  &=\exp\left\{\int_{0}^{1}t\left\langle n(x),\nabla \log f((1-u)\,x+u(x+t\,n(x)))\right\rangle du\right\}\\
  &=\exp\left\{\int_{0}^{1}t\left\langle n(x),\nabla \log f(x+ut\,n(x))\right\rangle du\right\}.
  \end{align*}
  Note that \(y=x+ut\,n(x)=(1+ut/\Vert x\Vert)\,x\). This implies that \(\Vert y\Vert=\Vert x\Vert +ut\). It can also be seen that  \(n(y)=y/\Vert y\Vert=x/\Vert x\Vert=n(x)\). Note that if \(\Vert x\Vert>M_s\), then \(\Vert y\Vert >M_s\) since $\|y\|=\|x\|+ut$. Hence 
  \begin{align*}
  \frac{f(x+t\, n(x))}{f(x)}&=\exp\left\{\int_{0}^{1}t\left\langle n(x),\nabla \log f(x+ut\,n(x))\right\rangle du\right\}\\
  &\le \exp\left\{\int_{0}^{1}t\left(C_1-f_s(\Vert x\Vert + ut)\right) du\right\}.
  \end{align*}
  Using the facts that $f_s(\cdot)$ is increasing and $t>0$,
  \begin{align*}
  \frac{f(x+t\, n(x))}{f(x)} &\le \exp\left\{\int_{0}^{1}t\left(C_1-f_s(\Vert x\Vert + ut)\right) du\right\}\\
  &\le \exp\left\{t\,C_1- t\,f_s(\Vert x\Vert)\right\}.
  \end{align*}
  Hence \(\Vert x\Vert >\max (f_s^{-1}(C_1-\log \epsilon/t),M_s)\),
  implies \(t\,C_1-t\, f_s(\Vert x\Vert)<\log \epsilon\). Hence we are done.
\end{proof}
\begin{lemma}\label{lemmadecay2}
Under assumption~\ref{assm:super'}, $f(x)$ is a decreasing function in every direction subject to $\|x\|>\max(f_s^{-1}(C_1),M_s)$.
\end{lemma}
\begin{proof}
    Using Lemma~\ref{lemmadecay1}, we have $\frac{f\left(x+t\frac{x}{\|x\|}\right)}{f(x)}<\epsilon$ whenever $\|x\|>\max(f_s^{-1}(C_1-\log \epsilon /t),M_s)$. Taking $\epsilon=1$  we have 
    $$\frac{f\left(x+t\frac{x}{\|x\|}\right)}{f(x)}<1$$ Hence we are done.
\end{proof}
Define $C_{f(x)}$ as the set $\left\{y:f(y)=f(x)\right\}$ for some fixed $x$. Also note that $f(x)$ has a maximum which is attained at some point of $\mathbb{R}^p$. This is easy to see using Lemma~\ref{lemmadecay2}. Denote this value by $p^*$. 
\begin{lemma}\label{lemmadecay3}
Under assumptions~\ref{assm:super'}, for $\|x\|>f_s^{-1}(C_1-\log (\epsilon/p^* \wedge 1))\vee M_s+1$,
$$f(x)<\epsilon.$$
\end{lemma}
\begin{proof}
    Using Lemma~\ref{lemmadecay1} we know that $\|x\|>f_s^{-1}(C_1-\frac{1}{t}\log \epsilon/p^*)$ implies 
    \begin{align*}
        \frac{f\left(x+t\frac{x}{\|x\|}\right)}{f(x)} \le \epsilon/p^*.
    \end{align*}
    Since \(p^*=\max_{x\in\mathbb{R}^p} f(x)\), this implies that $$f\left(x+t\frac{x}{\|x\|}\right)< \epsilon.$$
    Now this holds for any $t>0$. Consider $t=1$ and any $x$ such that $\|x\|>f_s^{-1}(C_1-\log (\epsilon/p^* \wedge 1))\vee M_s+1$. Define \(y=x-\frac{x}{\|x\|}\) which yields \(\|y\|=\|x\|(1-1/\Vert x\Vert)=\|x\|-1\). This also implies \(x=y+\frac{y}{\|y\|}\) as \[x=y+\frac{x}{\|x\|}\] and \(\frac{y}{\|y\|}=\frac{x(1-1/\|x\|)}{\|x\|-1}=\frac{x}{\|x\|}\). Now as $\|x\|>f_s^{-1}(C_1-\log (\epsilon/p^* \wedge 1))\vee M_s+1$,
    \[\left\Vert y\right\Vert+1 \ge f_s^{-1}(C_1-\log (\epsilon/p^* \wedge 1) )\vee M_s+1.\]This implies $\|y\|>f_s^{-1}(C_1-\log (\epsilon/p^* \wedge 1))\vee M_s$ and hence $\frac{f(y+\frac{y}{\|y\|})}{f(y)}<\epsilon/p^*$. This implies 
    \[\frac{f(x)}{f(y)} \le \epsilon/p^* \] which in turn implies 
    $f(x)<\epsilon$ as $p^*$ is maximum of the density $f(\cdot)$.
\end{proof}

\begin{lemma}\label{lemmapartition}
   Under assumptions~\ref{assm:super'} and~\ref{assm:envelope}, with $\|x\|>f_s^{-1}(C_1-\log (\exp\{- \tilde{f}(f_s^{-1}(C_1)\vee M_s)\}/p^*\wedge 1))\vee M_s+1$ , $C_{f(x)}$ partitions $\mathbb{R}^d$ into $A(x)$ and $R(x)$ with $C_{f(x)}$ as the boundary of both $A(x)$ and $R(x)$.
\end{lemma}
\begin{proof}
    Note that $f(x)$ is decreasing if $\|x\|>f_s^{-1}(C_1)\vee M_s$ by Lemma~\ref{lemmadecay2}. Consider the set $E=\left\{x: \|x\|\le f_s^{-1}(C_1)\vee M_s \right\}$ which happens to be compact.\\
    Now, by assumption~\ref{assm:envelope} $$\inf_{y \in E} f(y) \ge \exp\left\{-\tilde{f}(f_s^{-1}(C_1)\vee M_s)\right\}.$$ Hence for any $x$ such that $f(x)<\exp\left\{-\tilde{f}(f_s^{-1}(C_1)\vee M_s)\right\}$ , $C_{f(x)}$ will partition the space into $A(x)$ and $R(x)$. By Lemma~\ref{lemmadecay3} taking any $x$ such that $\|x\|>f_s^{-1}(C_1-\log (\exp\{- \tilde{f}(f_s^{-1}(C_1)\vee M_s)\}/p^*\wedge 1))\vee M_s+1$ we are done.
\end{proof}
We define a set on $\mathbb{R}^p$ that is fundamental to the next few steps. Define 
\[C_{f(x)}(\delta)=\left\{y+s\frac{y}{\|y\|}: y \in C_{f(x)}, -\delta \le s \le \delta \right\}.\]
\begin{lemma}\label{lemmaproposal11}
    Fix any $\epsilon>0$. Under assumption~\ref{assm:prop} with $\delta <\frac{\epsilon}{q(0)\, C_B(p) K_{\epsilon}^{p-1}}$ and \[\|x\|>2K^2_{\epsilon}\left[\left(\frac{\epsilon}{q(0)\, C_B(p) \delta}\right)^{1/p-1} -K_{\epsilon}\right]^{-1}+K_{\epsilon}\]
    we have \[Q(x,C_{f(x)}(\delta)\cap B(x,K_{\epsilon})) < \epsilon\]
    where $K_{\epsilon}$ is such that $Q(0,B(0,K_{\epsilon})^c)<\epsilon$.
\end{lemma}
\begin{proof}
   From previous works~{\citep[Theorem 4.1]{jarn:hans:2000}}, one observes
   \[\mu^{Leb}(C_{f(x)}(\delta)\cap B(x,K_{\epsilon}))\le \delta \,C_B(p) \,K_{\epsilon}^{p-1} \left(\frac{\|x\|+K_{\epsilon}}{\|x\|-K_{\epsilon}}\right)^{p-1}.\]
   Using this,
   \begin{align*}
       Q(x,C_{f(x)}(\delta)\cap B(x,K_{\epsilon})) = \int_{C_{f(x)}(\delta)\cap B(x,K_{\epsilon})} q(x,y) dy.
   \end{align*}
   Invoking assumption~\ref{assm:prop}, $q(0) \ge q(\|x\|)$ for any $x \in \mathbb{R}^p$. Hence
   \begin{align*}
       Q(x,C_{f(x)}(\delta)\cap B(x,K_{\epsilon})) &\le q(0) \int_{C_{f(x)}(\delta)\cap B(x,K_{\epsilon})}  dy \\
       & \le q(0) \,\mu^{Leb}(C_{f(x)}(\delta)\cap B(x,K_{\epsilon}))\\
       & \le q(0)\,\delta \,C_B(p) \,K_{\epsilon}^{p-1} \left(\frac{\|x\|+K_{\epsilon}}{\|x\|-K_{\epsilon}}\right)^{p-1}.
   \end{align*}
   Noting \[\|x\|>2K^2_{\epsilon}\left[\left(\frac{\epsilon}{q(0)\, C_B(p) \delta}\right)^{1/p-1} -K_{\epsilon}\right]^{-1}+K_{\epsilon}\]
   one immediately sees that 
   \[Q(x,C_{f(x)}(\delta)\cap B(x,K_{\epsilon})) <\epsilon.\]
   Also the condition on $\delta$ ensures the bound on $\|x\|$. Hence, we are done.
\end{proof}
\subsection{Main results for Section~\ref{sec:main_results}}
\begin{proposition}\label{prop2}
 Fix any $\epsilon>0$. Under assumptions~\ref{assm:super'}-\ref{assm:prop}, 
 \[\frac{PV(x)}{V(x)} <   Q(x,R(x)) + 8\epsilon\]
    if 
	\begin{align*}
	\|x\|&>\max\Bigg\{f^{-1}_s\left(C_1-\frac{1}{\delta}\log \epsilon\right),f^{-1}_s\left(C_1-\frac{1}{K_{\epsilon}}\log \epsilon\right),\\
	&\quad \quad \quad 2K^2_{\epsilon}\left[\left(\frac{\epsilon}{q(0)\, C_B(p) \delta}\right)^{1/p-1} -K_{\epsilon}\right]^{-1}+K_{\epsilon},\\
	& \quad \quad \quad f_s^{-1}\left(C_1-\log \left(\exp\left\{-\tilde f\left(f_s^{-1}\left(C_1\right)\vee M_s\right)\right\}/p^*\wedge 1\right)\right)+1,\ M_s+K_{\epsilon}\vee 1\Bigg\}.
	\end{align*}
\end{proposition}
\begin{proof}
    This is a fundamental results as it allows one to obtain a rate for the drift. The idea of the proof is to divide the space in to multiple sections with respect to the boundary of $A(x)$ and consider them separately. Consider any $x$ such that $C_{f(x)}$ partitions the space into $A(x)$ and $R(x)$. This is fundamental to the proof as it helps one break the cases in $A(x)$ and $R(x)$ by two categories i.e. points in $A(x)$ is of the form $z-s\frac{z}{\|z\|}$ for some $z$ such that $f(z)=f(x)$ and $s>0$. Similarly any point in $R(x)$ is of the form $z+s\frac{z}{\|z\|}$ such that $f(z)=f(x)$ and $s>0$. So,
    \[\|x\|>f_s^{-1}\left(C_1-\log \left(\exp\left\{-\tilde f\left(f_s^{-1}\left(C_1\right)\vee M_s\right)\right\}/p^*\wedge 1\right)\right)\vee M_s+1.\]
    It is easy to see that 
    $$\frac{PV(x)}{V(x)}=\int_{A(x)} q(x,y) \left(\frac{f(x)}{f(y)}\right)^{1/2} dy +\int_{R(x)} q(x,y) \left(1-\frac{f(y)}{f(x)}+\left(\frac{f(y)}{f(x)}\right)^{1/2}\right) dy.$$
    Therefore
    \begin{align*}
       \frac{PV(x)}{V(x)} \le  \int_{A(x)} q(x,y) \left(\frac{f(x)}{f(y)}\right)^{1/2} dy + Q(x,R(x)) + \int_{R(x)} q(x,y) \left(\frac{f(y)}{f(x)}\right)^{1/2} dy.
    \end{align*}
    Now, 
    \begin{align*}
        \int_{R(x)} q(x,y) \left(\frac{f(y)}{f(x)}\right)^{1/2} dy &=\int_{R(x)\cap C_{f(x)}(\delta) \cap B(x,K_{\epsilon})}q(x,y) \left(\frac{f(y)}{f(x)}\right)^{1/2} dy \\
       & \quad +\int_{R(x)\cap C_{f(x)}(\delta)\cap B(x,K_{\epsilon})^c} q(x,y) \left(\frac{f(y)}{f(x)}\right)^{1/2} dy\\
       & \quad +\int_{R(x)\cap C_{f(x)}(\delta)^c \cap B(x,K_{\epsilon})}q(x,y) \left(\frac{f(y)}{f(x)}\right)^{1/2} dy\\ &\quad +\int_{R(x)\cap C_{f(x)}(\delta)^c \cap B(x,K_{\epsilon})^c}q(x,y) \left(\frac{f(y)}{f(x)}\right)^{1/2} dy
    \end{align*}
    Note that on $R(x)$,$f(y)/f(x)<1$. Hence
    \begin{align*}
        \int_{R(x)\cap C_{f(x)}(\delta)\cap B(x,K_{\epsilon})^c} q(x,y) \left(\frac{f(y)}{f(x)}\right)^{1/2} dy &\le Q(x, B(x,K_{\epsilon})^c)\\
        &=Q(0,B(0,K_{\epsilon})\\
        &<\epsilon
    \end{align*}
    We can argue similarly for \[\int_{R(x)\cap C_{f(x)}(\delta)^c \cap B(x,K_{\epsilon})^c}q(x,y) \left(\frac{f(y)}{f(x)}\right)^{1/2} dy.\] 
    Now, for the first term,
    \begin{align*}
        \int_{R(x)\cap C_{f(x)}(\delta) \cap B(x,K_{\epsilon})}q(x,y) \left(\frac{f(y)}{f(x)}\right)^{1/2} dy  &\le Q(x,R(x)\cap C_{f(x)}(\delta) \cap B(x,K_{\epsilon}))\\
        &\le Q(x,C_{f(x)}(\delta) \cap B(x,K_{\epsilon})) \\
        &<\epsilon
    \end{align*}
    if $$\|x\|>2K^2_{\epsilon}\left[\left(\frac{\epsilon}{q(0)\, C_B(p) \delta}\right)^{1/p-1} -K_{\epsilon}\right]^{-1}+K_{\epsilon}$$
    where the last line follows from Lemma~\ref{lemmaproposal11}.
    For the remaining term note that if $y \in R(x)\cap C_{f(x)}(\delta)^c \cap B(x,K_{\epsilon})$, then there exists a $z$ such that $f(z)=f(x)$ and $y=z+s\frac{z}{\|z\|}$ such that $\delta <s< K_{\epsilon}$. 
    $$\frac{f(y)}{f(x)}=\frac{f(z+s\frac{z}{\|z\|})}{f(z)}<\epsilon$$
    if $$\|x\|>f^{-1}_s\left(C_1-\frac{1}{s}\log \epsilon\right)\vee M_s$$
    by Lemma~\ref{lemmadecay1}.
    
    Hence choosing
	\begin{align*}
	\|x\|&>\max\Bigg\{f^{-1}_s\left(C_1-\frac{1}{\delta}\log \epsilon\right),  f_s^{-1}\left(C_1-\log \left(\exp\left\{-\tilde f\left(f_s^{-1}\left(C_1\right)\vee M_s\right)\right\}/p^*\wedge 1\right)\right)+1,M_s+1\Bigg\}
	\end{align*}
 	implies 
	$$\int_{R(x)\cap C_{f(x)}(\delta)^c \cap B(x,K_{\epsilon})}q(x,y) \left(\frac{f(y)}{f(x)}\right)^{1/2} dy <\epsilon$$ which in turn gives 
	$$\int_{R(x)} q(x,y) \left(\frac{f(y)}{f(x)}\right)^{1/2} dy<4\epsilon$$
	if 
	\begin{align*}
	\|x\|&>\max\Bigg\{f^{-1}_s\left(C_1-\frac{1}{\delta}\log \epsilon\right),\ 2K^2_{\epsilon}\left[\left(\frac{\epsilon}{q(0)\, C_B(p) \delta}\right)^{1/p-1} -K_{\epsilon}\right]^{-1}+K_{\epsilon},\\
	& \quad \quad \quad f_s^{-1}\left(C_1-\log \left(\exp\left\{-f\left(f_s^{-1}\left(C_1\right)\vee M_s\right)\right\}/p^*\wedge 1\right)\right)+1,M_s+1\Bigg\}.
	\end{align*}
	One can similarly argue and conclude that 
	$$\int_{A(x)} q(x,y) \left(\frac{f(x)}{f(y)}\right)^{1/2} dy<4\epsilon$$
	under the same condition on $x$. 
	This is indeed true as the only case we have to consider is $A(x)\cap C_{f(x)}(\delta)^c \cap B(x,K_{\epsilon})$. For this set \(y=z-s\,z/\|z\|\) for some $z$ with $f(z)=f(x)$. Note that in this case as well $\delta <s<K_\epsilon$. Also note that for this particular case \[\frac{f(x)}{f(y)}=\frac{f(z)}{f(z-s\,z/\|z\|)}.\] Denote by $n(z)=z/\|z\|$. This implies 
	\begin{align*}
	\frac{f(x)}{f(y)}&=\exp\left\{\log(f(z))-\log(f(z-s\,n(z)))\right\}\\
	&=\exp\left\{\log(f(z))-\log(f(z))+s\int_{0}^{1}\left\langle \frac{z}{\|z\|},\nabla \log f(u\,z+(1-u)(z-s\,n(z)))\right\rangle du\right\}\\
	& \le \exp \left\{s\int_{0}^{1}\left(C_1-f_s(\|z-(1-u)s\,n(z)\|)\right) du\right\}\\
	&\le \exp \left\{s\,C_1-s\,\int_{0}^{1} f_s(\|z\|-(1-u)s) du\right\}\\
	&\le \exp \left\{s\left(C_1-f_s(\|z\|-s)\right)\right\}.
	\end{align*}
	Arguing as in Lemma~\ref{lemmadecay1}, this expression is less than $\epsilon$ when 
	\[\left\|z\right\|>f_s^{-1}\left(C_1-\frac{1}{\delta}\log\epsilon\right)\vee M_s+K_{\epsilon}.\]
	Hence 
	\begin{align*}
	\frac{PV(x)}{V(x)} <   Q(x,R(x)) + 8\epsilon
	\end{align*}
	if 
	\begin{align*}
	\|x\|&>\max\Bigg\{f_s^{-1}\left(C_1-\frac{1}{\delta}\log\epsilon\right)+K_{\epsilon},\ 2K^2_{\epsilon}\left[\left(\frac{\epsilon}{q(0)\, C_B(p) \delta}\right)^{1/p-1} -K_{\epsilon}\right]^{-1}+K_{\epsilon},\\
	& \quad \quad \quad f_s^{-1}\left(C_1-\log \left(\exp\left\{-f\left(f_s^{-1}\left(C_1\right)\vee M_s\right)\right\}/p^*\wedge 1\right) +\right)+1, \ M_s+K_{\epsilon}\vee 1\Bigg\}.
	\end{align*}
\end{proof}

\begin{proposition} 
\label{prop:R_eps}
 Suppose Assumptions~\ref{assm:cone}-\ref{assm:prop} hold, let $p^*=\max_{x}f(x)$,  and for any $\epsilon_{\alpha}$,
\begin{align}
    \epsilon\le \frac{1}{8}\frac{\epsilon_{\alpha}}{1+\epsilon_{\alpha}}\,Q\left(0,\left[-\frac{R_{\alpha}}{\sqrt{p}}-1,\frac{R_{\alpha}}{\sqrt{p}}-1\right]\times \left[-\frac{R_{\alpha}}{\sqrt{p}},\frac{R_{\alpha}}{\sqrt{p}}\right]^{p-1}\right).
\end{align}
Let
\[
\delta <\frac{\epsilon}{q(0)\,C_B(p)\, K^{p-1}_{\epsilon}},
\]
where $C_B(p)$ is is the Lebesgue volume of the unit ball and $K_{\epsilon}$ is such that $Q(0,B(0,K_{\epsilon})^c)\le \epsilon$ with $M^*=\max\{M_s,M_p\}$.
Then
\begin{equation}\label{convergenceradius}
\begin{split}
\mathcal{R}_{\epsilon}&=\max\Bigg\{f_s^{-1}\left(C_1-\frac{1}{\delta}\log\epsilon\right)+K_{\epsilon},\ 2K^2_{\epsilon}\left[\left(\frac{\epsilon}{q(0)\, C_B(p) \delta}\right)^{1/p-1} -K_{\epsilon}\right]^{-1}+K_{\epsilon},\\
& \quad \quad \quad f_s^{-1}\left(C_1-\log \left(\exp\left\{-\tilde f\left(f_s^{-1}\left(C_1\right)\vee M^*\right)\right\}/p^*\wedge 1\right)\right)+1, \ M^*+K_{\epsilon}\vee 1\Bigg\}.
\end{split}
\end{equation}
Also   $$\left\{x:V(x)<d\right\} \subset \overline{B(0,\mathcal{R}_{\max})}$$ for $d=(2b+\epsilon_{\alpha})(1+\tilde\lambda)^{-1}$
  where \[\mathcal{R}_{\max}=f_s^{-1}\left(C_1-\log\left(\frac{(1-\tilde\lambda)^2}{2\,p^* \,(2b+\epsilon_{\alpha})^2}\wedge 1\right)\right)\vee M^*.\]

\end{proposition}

\begin{proof}[Proof of Proposition~\ref{prop:R_eps}]
This is the main proposition for the calculation of the upper bound for the rate of convergence. This result quantifies $\mathcal{R}_{\epsilon}$ and $\mathcal{R}_{\max}$ which were introduced in Propositions~\ref{prop:driftJH} and~\ref{prop:minorJH}.
Note that by Proposition~\ref{prop2},
	\begin{align*}
	\|x\|&>\max\Bigg\{f_s^{-1}\left(C_1-\frac{1}{\delta}\log\epsilon\right)+K_{\epsilon},\ 2K^2_{\epsilon}\left[\left(\frac{\epsilon}{q(0)\, C_B(p) \delta}\right)^{1/p-1} -K_{\epsilon}\right]^{-1}+K_{\epsilon},\\
	& \quad \quad \quad f_s^{-1}\left(C_1-\log \left(\exp\left\{-f\left(f_s^{-1}\left(C_1\right)\vee M_s\right)\right\}/p^*\wedge 1\right) +\right)+1, \ M_s+1\Bigg\}.
	\end{align*}
implies
     \begin{align*}
       \frac{PV(x)}{V(x)} <   Q(x,R(x)) + 8\epsilon
    \end{align*}
We know $Q(x,A(x))\ge Q(x,W(x))$ from assumption~\ref{assm:cone}.
 By Lemma~\ref{volumelemma} we know that
 	\begin{align*}
	Q(x,W(x))\ge Q\left(0,\left[-\frac{R_{\alpha}}{\sqrt{p}}-1,\frac{R_{\alpha}}{\sqrt{p}}-1\right]\times \left[-\frac{R_{\alpha}}{\sqrt{p}},\frac{R_{\alpha}}{\sqrt{p}}\right]^{p-1}\right)>0,
	\end{align*}
	Thus using the proof of Proposition~\ref{prop:driftJH},
	\[\frac{PV(x)}{V(x)}\le \tilde \lambda\]
	if 
	\begin{align*}
	\|x\|&>\max\Bigg\{f_s^{-1}\left(C_1-\frac{1}{\delta}\log\epsilon\right)+K_{\epsilon},\ 2K^2_{\epsilon}\left[\left(\frac{\epsilon}{q(0)\, C_B(p) \delta}\right)^{1/p-1} -K_{\epsilon}\right]^{-1}+K_{\epsilon},\\
	& \quad \quad \quad f_s^{-1}\left(C_1-\log \left(\exp\left\{-f\left(f_s^{-1}\left(C_1\right)\vee M_s\right)\right\}/p^*\wedge 1\right) +\right)+1, \ M_s+1,\, M_p\Bigg\}.
	\end{align*}
Therefore 
\begin{align*}
    \mathcal{R}_{\epsilon}&=\max\Bigg\{f_s^{-1}\left(C_1-\frac{1}{\delta}\log\epsilon\right)+K_{\epsilon},\ 2K^2_{\epsilon}\left[\left(\frac{\epsilon}{q(0)\, C_B(p) \delta}\right)^{1/p-1} -K_{\epsilon}\right]^{-1}+K_{\epsilon},\\
	& \quad \quad \quad f_s^{-1}\left(C_1-\log \left(\exp\left\{-f\left(f_s^{-1}\left(C_1\right)\vee M^*\right)\right\}/p^*\wedge 1\right) +\right)+1, \ M^*+K_{\epsilon}\vee 1\Bigg\}
\end{align*}
suffices.
Hence we are done with the first part.
     As in Proposition~\ref{prop:minorJH} take $d=\frac{2b+\epsilon_{\alpha}}{(1-\tilde\lambda)}$.  From the proof of Proposition~\ref{prop:minorJH}, one observes
     \begin{align*}
         \left\{x: V(x)<\frac{2b+\epsilon_{\alpha}}{(1-\tilde\lambda)}\right\}=\left\{x:\frac{(1-\tilde\lambda)^2}{(2b+\epsilon_{\alpha})^2}<f(x)\right\} \subset \left\{x:\frac{(1-\tilde\lambda)^2}{2(2b+\epsilon_{\alpha})^2}\le f(x)\right\}.
     \end{align*}
 	Note that by Lemma~\ref{lemmadecay3}  $\|x\|>f_s^{-1}\left(C_1-\log\left(\frac{(1-\tilde\lambda)^2}{2\,p^* \,(2b+\epsilon_{\alpha})^2}\wedge 1\right)\right)\vee M^*$ implies $f(x)<\frac{(1-\tilde\lambda)^2}{2(2b+\epsilon_{\alpha})^2}$ as $M^*\ge M_s$.
	Therefore
	\begin{align*}
	\left\{x:\|x\|>f_s^{-1}\left(C_1-\log\left(\frac{(1-\tilde\lambda)^2}{2\,p^* \,(2b+\epsilon_{\alpha})^2}\wedge 1\right)\right)\vee M^*\right\} \subset \left\{x: f(x)<\frac{(1-\tilde\lambda)^2}{2(2b+\epsilon_{\alpha})^2} \right\}.
	\end{align*}
	This in turn implies 
	\begin{align*}
	\left\{x: f(x)\ge \frac{(1-\tilde\lambda)^2}{2(2b+\epsilon_{\alpha})^2} \right\} \subset \left\{x:\|x\|\le f_s^{-1}\left(C_1-\log\left(\frac{(1-\tilde\lambda)^2}{2\,p^* \,(2b+\epsilon_{\alpha})^2}\wedge 1\right)\right)\vee M^*\right\}=\overline{B(0,\mathcal{R}_{\max})}
	\end{align*}
	with \[\mathcal{R}_{\max}=f_s^{-1}\left(C_1-\log\left(\frac{(1-\tilde\lambda)^2}{2\,p^* \,(2b+\epsilon_{\alpha})^2}\wedge 1\right)\right)\vee M^*.\]
 \end{proof}

\begin{remark}
By convention, if $x\notin Range(f_s)$, i.e., $x< \inf Range(f_s)$, then $f_s^{-1}(x)=\inf Domain(f_s)$.
\end{remark}

\subsection{Proposition~\ref{prop:more_lowerbounds}}
\label{sec:more_lb}
Suppose there exists $B(x) < \infty$ such that $q(x,y) \le B(x)$.  Then
\begin{equation*}
 \int_{\mathbb{R}^p} \alpha(x,y) q(x,y) dy   = \int_{\mathbb{R}^p} \left[ \frac{q(x,y)}{f(x)} \wedge \frac{q(x,y)}{f(y)}\right] f(y) dy
  \le \frac{B(x)}{f(x)} 
\end{equation*}
and the claim follows from Theorem~\ref{thm:main_result}.

If $f$ has at most countably many modes, then for mode $x^*$,
\begin{align*}
\int_{\mathbb{R}^p} \alpha(x^*,y) q(x^*,y) dy & =\int_{f(y)<f(x^*)} \frac{f(y)}{f(x^*)} q(x^*,y) dy +\int_{f(y)=f(x^*)} q(x^*,y) dy \\
    &= \int \frac{f(y)}{f(x^*)} q(x^*,y) dy
\end{align*}
and the claim follows from Theorem~\ref{thm:main_result}.

\subsection{Proof of Proposition~\ref{prop:class}}
\label{app:prop_class}

We analyze the two cases separately and establish assumptions~\ref{assm:cone}, ~\ref{assm:super'} and~\ref{assm:envelope} for both cases. Note that the aforementioned assumptions are sufficient for the upper bound of the rate of the convergence of RWMH algorithm subject to the proposal satisfying assumption~\ref{assm:prop}.
    Consider the first case with $f=f_1\cdot f_2$. This implies $\log f=\log f_1+\log f_2$. Now, 
    \begin{align*}
        \left|\log f\right|&\le \left|\log f_1\right|+ \left|\log f_2\right|\\
        &\le \tilde{f}_1(\|x\|)+\tilde{f}_2(\|x\|)
    \end{align*}
    where $\tilde{f}_1,\tilde{f}_2$ are the functions as in assumption~\ref{assm:envelope} for $f_1, f_2$ respectively.
    Note that since $f_1$ and $f_2$ satisfy  assumption~\ref{assm:super'}, 
    \[\left\langle \frac{x}{\left\|x\right\|}, \nabla \log f_1\right\rangle \le C^*_1 -f_{1s}(\|x\|)\]
    and 
    \[\left\langle \frac{x}{\left\|x\right\|}, \nabla \log f_2\right\rangle \le C^*_2 -f_{2s}(\|x\|)\]
    for constants $C^*_1$ and $C^*_2$ and $f_{1s}, f_{2s}$ continuous increasing functions.
    This implies 
    \begin{align*}
        \left\langle \frac{x}{\left\|x\right\|}, \nabla \log f\right\rangle &\le \left\langle \frac{x}{\left\|x\right\|}, \nabla \log f_1\right\rangle +\left\langle \frac{x}{\left\|x\right\|}, \nabla \log f_2\right\rangle \\
        &\le C^*_1 -f_{1s}(\|x\|)+C^*_2 -f_{2s}(\|x\|)
    \end{align*}
    Hence $f$ satisfies assumption~\ref{assm:super'} with $C^*_3=C^*_1+C^*_2$ and $f_s=f_{1s}+f_{2s}$. Lastly note that since $f_1$ and $f_2$ satisfy \eqref{eq:curvature:exact}, 
    \begin{align*}
        \left\langle\frac{x}{\|x\|},\frac{\nabla f(x)}{\left\|\nabla f(x)\right\|}\right\rangle&=       \left\langle\frac{x}{\|x\|},\frac{\nabla \log f(x)}{\left\|\nabla \log f(x)\right\|}\right\rangle\\
        &=\left\langle\frac{x}{\|x\|},\frac{\nabla \log f_1(x)+\log f_2(x)}{\left\|\nabla \log f_1(x)+\log f_2(x)\right\|}\right\rangle\\
        &=\left\langle \frac{x}{\|x\|}, \frac{\nabla \log f_1(x)}{\left\|\nabla \log f_1(x)\right\|}\right\rangle \frac{\left\|\nabla \log f_1(x)\right\|}{\left\|\nabla \log f_1(x)+\log f_2(x)\right\|}\\
        &\quad +\left\langle \frac{x}{\|x\|}, \frac{\nabla \log f_2(x)}{\left\|\nabla \log f_2(x)\right\|}\right\rangle \frac{\left\|\nabla \log f_2(x)\right\|}{\left\|\nabla \log f_1(x)+\log f_2(x)\right\|}.
    \end{align*}
    For $x$ with $|x|>M_{1p}\vee M_{2p}$,  
    \begin{align*}
         \left\langle\frac{x}{\|x\|},\frac{\nabla f(x)}{\left\|\nabla f(x)\right\|}\right\rangle &\le -\eta_1 \, \frac{\left\|\nabla \log f_2(x)\right\|}{\left\|\nabla \log f_1(x)+\nabla \log f_1(x)\right\|}   -\eta_2\, \frac{\left\|\nabla \log f_2(x)\right\|}{\left\|\nabla \log f_1(x)+\nabla \log f_2(x)\right\|} \\
         &\le -\eta_1\wedge\eta_2
    \end{align*}
For the second case, note that 
\[\left|\log \left(a_1\,f_1+a_2\,f_2\right)\right|\le \log \left(\exp\left(\tilde{f}_1+\tilde{f}_2\right)+1\right).\]
Hence assumption~\ref{assm:envelope} is satisfied.
Also, 
\begin{align*}
    \left\langle \frac{x}{\left\|x\right\|},\nabla \log f(x)\right\rangle &=    \left\langle \frac{x}{\left\|x\right\|},\nabla \log f(x)\right\rangle\\
    &=\left\langle \frac{x}{\left\|x\right\|},\nabla \log \left(a_1\,f_1(x)+a_2\,f_2(x)\right)\right\rangle\\
    &=\left\langle \frac{x}{\left\|x\right\|},\frac{a_1\,\nabla f_1(x)+a_2\,\nabla f_2(x)}{a_1\,f_1(x)+a_2\,f_2(x)}\right\rangle\\
    &=\frac{a_1\, f_1(x)}{a_1\,f_1(x)+a_2\,f_2(x)}\left\langle\frac{x}{\left\|x\right\|},\frac{\nabla f_1(x)}{f_1(x)}\right\rangle +\frac{a_2\, f_2(x)}{a_1\,f_1(x)+a_2\,f_2(x)}\left\langle\frac{x}{\left\|x\right\|},\frac{\nabla f_2(x)}{f_2(x)}\right\rangle\\
    &=\frac{a_1\, f_1(x)}{a_1\,f_1(x)+a_2\,f_2(x)}\left\langle\frac{x}{\left\|x\right\|},\nabla \log f_1(x)\right\rangle +\frac{a_2\, f_2(x)}{a_1\,f_1(x)+a_2\,f_2(x)}\left\langle\frac{x}{\left\|x\right\|},\nabla \log f_2(x)\right\rangle\\
    &\le \frac{a_1\, f_1(x)}{a_1\,f_1(x)+a_2\,f_2(x)} \left(C^*_1-f_{1s}(\|x\|)\right)+\frac{a_2\, f_2(x)}{a_1\,f_1(x)+a_2\,f_2(x)}\left(C^*_2-f_{2s}(\|x\|)\right)\\
    &\le C^*_1\vee C^*_2-f_{1s}\wedge f_{2s}(\|x\|)
\end{align*}
It is easy to see that $f_s=f_{1s}\wedge f_{2s}$ is an increasing continuous function.
Finally, one has for $x$ with $\|x\|>M_{1p}\vee M_{2p}$,
\begin{align*}
    \left\langle\frac{x}{\left\|x\right\|},\frac{\nabla f(x)}{\left\|\nabla f(x)\right\|}\right\rangle&=    \left\langle\frac{x}{\left\|x\right\|},\frac{ a_1\,\nabla f_1(x)+a_2\,\nabla f_2(x)}{\left\| a_1\,\nabla f_1(x)+a_2\,\nabla f_2(x)\right\|}\right\rangle\\
    &=\frac{\left\|a_1\, \nabla f_1(x)\right\|}{\left\| a_1\,\nabla f_1(x)+a_2\,\nabla f_2(x)\right\|}\left\langle\frac{x}{\left\|x\right\|},\frac{\nabla f_1(x)}{\left\|\nabla f_1(x)\right\|}\right\rangle \\
    &\quad +\frac{\left\|a_2\, \nabla f_2(x)\right\|}{\left\| a_1\,\nabla f_1(x)+a_2\,\nabla f_2(x)\right\|}\left\langle\frac{x}{\left\|x\right\|},\frac{\nabla f_2(x)}{\left\|\nabla f_2(x)\right\|}\right\rangle\\
    &\le -\eta_1\wedge \eta_2.
\end{align*}
Hence we are done.

\section{Proofs for GLM Type Families}
 \begin{lemma}\label{glmsuplem}
	Let Assumptions~\ref{assm:glm:diff}-\ref{assm:glm:cumulant} and Assumptions~\ref{assm:prior:cnvx} or Assumption~\ref{assm:prior:diss} hold. On the set $|\theta|\ge 1$, the family of distributions with density as defined in \eqref{glmposterior} is superexponentially light with rate function \[f_s(x)= c x\]  where $c=\lambda_1$  under Assumption~\ref{assm:prior:cnvx} and $c=a_{\dag}$ under Assumption~\ref{assm:prior:diss}. 
\end{lemma}
\begin{proof}[Proof of Lemma~\ref{glmsuplem}]
	The proof is based on finding upper bound estimates for the posterior as defined in \eqref{glmposterior} by dividing the problem into two parts based on the prior being strongly convex or dissipative. First consider the case where $g(\cdot)$ is strongly convex. Note that this shall have two sub-cases based on whether $c(\theta,x)$ is bounded in $\theta$ or whether $c(\theta,x)$ is convex with bounded curvature. Consider the first sub-case. Note that 
	\begin{align*}
	\nabla \log f(\theta \mid y,x)=\sum_{i=1}^{n}\nabla \Pi(\theta,x_i)\cdot T(y_i)-\sum_{i=1}^{n}\nabla c(\theta,x_i)-\nabla g(\theta).
	\end{align*}
	Hence,
	\begin{align*}
	&\frac{1}{\left\|\theta\right\|}\theta^{\mathsf{T}}\nabla \log f(\theta \mid y,x)\\&=\frac{1}{\left\|\theta\right\|}\sum_{i=1}^{n}\theta^{\mathsf{T}}\nabla \Pi(\theta,x_i)\cdot T(y_i)-\frac{1}{\left\|\theta\right\|}\sum_{i=1}^{n}\theta^{\mathsf{T}}\nabla c(\theta,x_i)-\frac{1}{\left\|\theta\right\|}\theta^{\mathsf{T}}\nabla g(\theta)\\
	&\le \lambda(x_1,x_2,\cdots,x_n) \sum_{i=1}^{n}\left\|T(y_i)\right\|+n\, K(x_1,x_2,\cdots,x_n)+\left\|\nabla g(0)\right\|-\frac{1}{\|\theta\|}\theta^{\mathsf{T}}\nabla^2 g(\xi_{\theta})\theta.
	\end{align*}	
	Note that since $g$ is strongly convex , $\lambda_1>0$ such that \(\theta^{\mathsf{T}}\nabla^2 g(x)\theta >\lambda_1\) for all $x$ and $\theta$ such that $\|\theta\|=1$. Therefore,
		\begin{align*}
		\frac{1}{\left\|\theta\right\|}\theta^{\mathsf{T}}\nabla \log f(\theta \mid y,x)
		&\le \lambda(x_1,x_2,\cdots,x_n) \sum_{i=1}^{n}\left\|T(y_i)\right\|+n\, K(x_1,x_2,\cdots,x_n)+\left\|\nabla g(0)\right\|-\lambda_1 \left\|\theta\right\|.
		\end{align*}
	Therefore the first sub-case is established. Now, consider the case where $c(\cdot,x)$ is convex with bounded curvature. In this regime,
	\begin{align*}
	&\frac{1}{\left\|\theta\right\|}\theta^{\mathsf{T}}\nabla \log f(\theta \mid y,x)\\
	&\le \lambda(x_1,x_2,\cdots,x_n) \sum_{i=1}^{n}\left\|T(y_i)\right\|-\sum_{i=1}^{n}\frac{1}{\left\|\theta\right\|}\theta^{\mathsf{T}}\left(\nabla c(0,x_i)+\nabla^2 c(\xi_{\theta,x},x)\theta\right)-\frac{1}{\left\|\theta\right\|} \theta^{\mathsf{T}}\nabla g(\theta)\\
	&\le  \lambda(x_1,x_2,\cdots,x_n) \sum_{i=1}^{n}\left\|T(y_i)\right\| +n\max_{1\le i\le n} \left\|\nabla c(0,x_i)\right\|+\left\|\nabla g(0)\right\|\\
	&\quad \quad -\frac{1}{\|\theta\|}\theta^{\mathsf{T}}\left(\nabla^2 c(\xi_{\theta,x},x)+\nabla^2 g(\xi_{\theta})\right)\theta\\
	&\le \lambda(x_1,x_2,\cdots,x_n) \sum_{i=1}^{n}\left\|T(y_i)\right\| +n\max_{1\le i\le n} \left\|\nabla c(0,x_i)\right\|+\left\|\nabla g(0)\right\|-\lambda_1\left\|\theta\right\|.
	\end{align*} 
	Next consider the case where $g(\cdot)$ is $(a,b)$-dissipative where divide the problem again, in two cases- where $c(\cdot,x)$ is bounded, and $c(\cdot,x)$ is convex and the curvature is bounded. In the first regime,
	\begin{align*}
	\frac{1}{\left\|\theta\right\|}\theta^{\mathsf{T}}\nabla \log f(\theta \mid y,x)&=\frac{1}{\left\|\theta\right\|}\sum_{i=1}^{n}\theta^{\mathsf{T}}\nabla \Pi(\theta,x_i)\cdot T(y_i)-\frac{1}{\left\|\theta\right\|}\sum_{i=1}^{n}\theta^{\mathsf{T}}\nabla c(\theta,x_i)-\frac{1}{\left\|\theta\right\|}\theta^{\mathsf{T}}\nabla g(\theta)\\
	&\le \lambda(x_1,x_2,\cdots,x_n) \sum_{i=1}^{n}\left\|T(y_i)\right\|+n\, K(x_1,x_2,\cdots,x_n)-a_{\dag}\left\|\theta\right\|+\frac{b_{\dag}}{\left\|\theta\right\|}\\
	&\le \lambda(x_1,x_2,\cdots,x_n) \sum_{i=1}^{n}\left\|T(y_i)\right\|+n\, K(x_1,x_2,\cdots,x_n)-a_{\dag}\left\|\theta\right\|+b_{\dag}.
	\end{align*}
	In the regime where $c(\cdot,x)$ is convex with bounded curvature,
	\begin{align*}
	\frac{1}{\left\|\theta\right\|}\theta^{\mathsf{T}}\nabla \log f(\theta \mid y,x)&=\frac{1}{\left\|\theta\right\|}\sum_{i=1}^{n}\theta^{\mathsf{T}}\nabla \Pi(\theta,x_i)\cdot T(y_i)-\frac{1}{\left\|\theta\right\|}\sum_{i=1}^{n}\theta^{\mathsf{T}}\nabla c(\theta,x_i)-\frac{1}{\left\|\theta\right\|}\theta^{\mathsf{T}}\nabla g(\theta)\\
	&\le \lambda(x_1,x_2,\cdots,x_n) \sum_{i=1}^{n}\left\|T(y_i)\right\|+n\, \max_{1\le i\le n}\left\|\nabla c(0,x_i)\right\|-a_{\dag}\left\|\theta\right\|+\frac{b_{\dag}}{\left\|\theta\right\|}\\
	&\le \lambda(x_1,x_2,\cdots,x_n) \sum_{i=1}^{n}\left\|T(y_i)\right\|+n\, \max_{1\le i\le n}\left\|\nabla c(0,x_i)\right\|-a_{\dag}\left\|\theta\right\|+b_{\dag}.
	\end{align*}
	Hence we are done.		
\end{proof}
Note that $C_1$ has a different expression in four separate cases as- bounded derivative and strongly convex $g(\cdot)$,  convex $c(\cdot,x)$ and strongly convex $g(\cdot)$, bounded derivative and dissipative $g(\cdot)$, convex $c(\cdot,x)$ and dissipative $g(\cdot)$ as
\begin{align*}
        C_1&=\lambda(x_1,x_2,\cdots,x_n) \sum_{i=1}^{n}\left\|T(y_i)\right\|+n\, K(x_1,x_2,\cdots,x_n)+\left\|\nabla g(0)\right\|\\
        C_1&=\lambda(x_1,x_2,\cdots,x_n) \sum_{i=1}^{n}\left\|T(y_i)\right\| +n\max_{1\le i\le n} \left\|\nabla c(0,x_i)\right\|+\left\|\nabla g(0)\right\|\\
        C_1&=\lambda(x_1,x_2,\cdots,x_n) \sum_{i=1}^{n}\left\|T(y_i)\right\|+n\, K(x_1,x_2,\cdots,x_n)+b_{\dag}\\
        C_1&=\lambda(x_1,x_2,\cdots,x_n) \sum_{i=1}^{n}\left\|T(y_i)\right\| +n\max_{1\le i\le n} \left\|\nabla c(0,x_i)\right\|+b_{\dag}
\end{align*}
respectively.
\begin{lemma}\label{lemmabounded1}
	Let the Assumptions~\ref{assm:glm:diff}-~\ref{assm:prior:Lipschitz} hold. Then for the posterior  $f(\theta\mid y,x)$, as defined in \eqref{glmposterior}, there exists a quadratic function $f_{y,x}:\mathbb{R}\to \mathbb{R}_{+}$ such that
	\begin{align*}
	\left|\log f(\theta\mid y,x)\right|\le f_{y,x}(\left\|\theta\right\|),
	\end{align*} 
	for all $\theta \in \mathbb{R}^p$ where 
	\begin{align*}
	    f_{y,x}(\left\|\theta\right\|)=K_1+K_2\left\|\theta\right\|+K_3\left\|\theta\right\|^2.
	\end{align*}
	For  the case where $\nabla c(\theta,x)$ is bounded
	\begin{align*}
	    K_1&=\max_{1\le i \le n} \left\|\Pi(0,x_i)\right\|\sum_{i=1}^{n} \left\|T(y_i)\right\|+n\max_{1\le i \le n}\left\|c(0,x_i)\right\|+\left\|g(0)\right\|,\\
	    K_2&=\lambda(x_1,x_2,\cdots,x_n)\sum_{i=1}^{n} \left\|T(y_i)\right\|+K(x_1,x_2\cdots,x_n)+\left\|\nabla g(0)\right\|,\\
	    K_3&=\frac{1}{2}\lambda_2.
	\end{align*}
	 and for the case that $c(\theta,x)$ is convex in $\theta$ and $\nabla^2 c(\theta,x)$ is bounded
		\begin{align*}
	    K_1&=\max_{1\le i \le n} \left\|\Pi(0,x_i)\right\|\sum_{i=1}^{n} \left\|T(y_i)\right\|+n\max_{1\le i \le n}\left\|c(0,x_i)\right\|+\left\|g(0)\right\|\\
	    K_2&=\lambda(x_1,x_2,\cdots,x_n)\sum_{i=1}^{n} \left\|T(y_i)\right\|+\max_{1\le i \le n}\left\|\nabla c(0,x_i)\right\|+\left\|\nabla g(0)\right\|,\\
	    K_3&=\frac{1}{2}\left(\lambda_2+K(x_1,x_2,\cdots,x_n)\right).
	\end{align*}
.
\end{lemma}

\begin{proof}[Proof of Lemma~\ref{lemmabounded1}]
	 This result allows one to establish an envelope function for \eqref{glmposterior}, with certain monotone property. Firstly, one observes
	\begin{align*}
	\exp\left\{\sum_{i=1}^{n}T(y_i)\cdot \Pi(\theta,x_i)-\sum_{i=1}^{n}c(\theta,x_i)-g(\theta)\right\}\le \exp\left\{\sum_{i=1}^{n}\left|\Pi(\theta,x_i)\right|\left|T(y_i)\right| +\sum_{i=1}^{n}\left|c(\theta,x_i)\right|+\left|g(\theta)\right|\right\},
	\end{align*}
	which follows from the monotonicity of the exponential function.
	We have, from Assumptions~\ref{assm:glm:diff}-\ref{assm:prior:Lipschitz},
	\begin{align*}
	\left\|\Pi(\theta,x_i)\right\|&\le \left\|\Pi(\theta,x_i)-\Pi(0,x_i)\right\|+\left\|\Pi(0,x_i)\right\|\\
	& =\left\|\nabla \Pi(\xi,x_i)\theta\right\|+\left\|\Pi(0,x_i)\right\|\\
	& \le \left|\left|\nabla \Pi(\xi,x_i)\right|\right|_{sp}\left\|\theta\right\|+\left\|\Pi(0)\right\|\\
	&\le \lambda(x_1,x_2,\cdots,x_n) \left\|\theta\right\| +\left\|\Pi(0,x_i)\right\|\\
	&\le \lambda(x_1,x_2,\cdots,x_n) \left\|\theta\right\|+\max_{1\le i \le n} \left\|\Pi(0,x_i)\right\|.
	\end{align*}
	Similarly, for the function $c(\theta)$, under the first case, we have
	\begin{align*}
	\left|c(\theta,x_i)\right|&\le \left\|\nabla c(\xi,x_i)\theta\right\|+\left|c(0,x_i)\right|\\
	& \le K(x_1,x_2,\cdots,x_n)\left\|\theta\right\|+\left|c(0,x_i)\right|\\
	&\le K(x_1,x_2,\cdots,x_n)\left\|\theta\right\|+\max_{1\le i \le n}\left|c(0,x_i)\right|,
	\end{align*}
	and under the second case where $c(\theta,x)$ is convex with bounded curvature, 
	\begin{align*}
	\left|c(\theta,x_i)\right|&\le \left|c(0,x_i)\right|+\left\|\nabla c(0,x_i)^{T}\theta\right\|+\frac{1}{2}\theta^{T} \nabla^2c(\xi,x_i)\theta\\
	&\le \max_{1\le i \le n}\left|c(0,x_i)\right|+\left\|\nabla c(0,x_i)\right\|\left\|\theta\right\|+\frac{K(x_1,x_2,\cdots,x_n)}{2}\left\|\theta\right\|^2.
	\end{align*}
	Note that $\theta^{T} \nabla^2c(\xi,x_i)\theta\ge 0$ by the convexity of $c$.

	Also,
	\begin{align*}
	\left|g(\theta)\right|&=\left|g(0)+\theta\cdot \nabla g(0)+\frac{1}{2}\theta^{T}\nabla^2 g(\xi_2)\theta\right|\\
	&\le \left|g(0)\right|+\left\|\nabla g(0)\right\|\left\|\theta\right\|+\frac{1}{2}\lambda_2 \left\|\theta\right\|^2.
	\end{align*}
		In the case where \(\|\nabla c(\theta,x_i)\| \le K(x_1,x_2,\cdots,x_n)\), define
	\begin{align*}
	    K_1&=\max_{1\le i \le n} \left\|\Pi(0,x_i)\right\|\sum_{i=1}^{n} \left\|T(y_i)\right\|+n\max_{1\le i \le n}\left|c(0,x_i)\right|+\left|g(0)\right|,\\
	    K_2&=\lambda(x_1,x_2,\cdots,x_n)\sum_{i=1}^{n} \left\|T(y_i)\right\|+K(x_1,x_2\cdots,x_n)+\left\|\nabla g(0)\right\|,\\
	    K_3&=\frac{1}{2}\lambda_2.
	\end{align*}
	In the second case, where $c(\cdot,x_i)$ is convex for all $i$ with $\nabla^2 c(\theta,x_i) \le K(x_1,x_2,\cdots,x_n)$, define
	\begin{align*}
	    K_1&=\max_{1\le i \le n} \left\|\Pi(0,x_i)\right\|\sum_{i=1}^{n} \left\|T(y_i)\right\|+n\max_{1\le i \le n}\left|c(0,x_i)\right|+\left|g(0)\right|,
	\end{align*}
	\begin{align*}
	    K_2&=\lambda(x_1,x_2,\cdots,x_n)\sum_{i=1}^{n} \left\|T(y_i)\right\|+\max_{1\le i \le n}\left\|\nabla c(0,x_i)\right\|+\left\|\nabla g(0)\right\|,\\
	    K_3&=\frac{1}{2}\left(\lambda_2+K(x_1,x_2,\cdots,x_n)\right).
	\end{align*}
	Observe that taking 
	\begin{align*}
	    f_{y,x}(\left\|\theta\right\|)=K_1+K_2\left\|\theta\right\|+K_3\left\|\theta\right\|^2
	\end{align*}
	suffices. Hence we are done.
\end{proof}
\begin{remark}
    Lemma~\ref{lemmabounded1} exhibits that we satisfy assumption~\ref{assm:envelope} for the posterior defined in \eqref{glmposterior}.
\end{remark}

\begin{lemma}\label{glmmplemma}
	Let Assumptions~\ref{assm:glm:diff}-~\ref{assm:prior:Lipschitz} hold and let either Assumption~\ref{assm:prior:cnvx} or Assumption~\ref{assm:prior:diss} be true. Define $\gamma=\lambda_1$  when Assumption~\ref{assm:prior:cnvx} holds and $\gamma=a_{\dag}$ when Assumption~\ref{assm:prior:diss} holds with $0<\eta<\gamma/\lambda_2$. Also define 
	\[\tilde{J}(x_1,x_2,\cdots,x_n)=n\max_{1\le i \le n}\left\|\nabla c(0,x_i)\right\|+n\, K(x_1,x_2,\cdots,x_n)\] when $c(\theta,x)$ is convex in $\theta$ with bounded Hessian and 
	\[\tilde{J}(x_1,x_2,\cdots,x_n)=n\, K(x_1,x_2,\cdots,x_n)\]
	when $c(\theta,x)$ is  is bounded.
	In this setting,
	\[f_s (\left\|\theta\right\|)-C_1 \ge \eta \left\|\nabla f(\theta)\right\| \]
	with \[M'_p=\max\left[\frac{1}{\gamma-\eta \lambda_2} \left(C_1+\lambda \sum_{i=1}^{n} \left\|T(y_i)\right\|+\tilde{J}(x_1,x_2,\cdots,x_n)+\left\|\nabla g(0)\right\|\right),1\right].\]
 Also $C_1$ is as defined in \eqref{eq:C_1_1}-\eqref{eq:C_1_4}.
\end{lemma}

\begin{proof}[Proof of Lemma~\ref{glmmplemma}]
     We only consider the case where $c(\theta,x)$ has bounded curvature with respect to $\theta$ since the other case is exactly similar. 
     \[\nabla \log f(\theta\mid y,x)=\sum_{i=1}^{n}\nabla \Pi(\theta,x_i)\cdot T(y_i)-\sum_{i=1}^{n}\nabla c(\theta,x_i)-\nabla g(\theta).\]
     Hence, in the case of strong convexity of the objective function $g(\cdot)$, we have  
     \begin{align*}
         \left\|\nabla \log f(\theta\mid y,x)\right\|&\le \sum_{i=1}^{n}\left\|\nabla \Pi(\theta,x_i)\cdot T(y_i)\right\|+\sum_{i=1}^{n} \left\|\nabla c(\theta,x_i)\right\|+\left\|\nabla g(\theta)\right\|\\
         & \le \sum_{i=1}^{n}\|\nabla \Pi(\theta,x_i)\|_{sp} \left\|T(y_i)\right\|+\sum_{i=1}^{n} \left\|\nabla c(\theta,x_i)\right\|+\left\|\nabla g(\theta)\right\|\\
          &\le \lambda \sum_{i=1}^{n} \left\|T(y_i)\right\|+\sum_{i=1}^{n} \left\|\nabla c(\theta,x_i)\right\|+\left\|\nabla g(0)\right\|+\left\|\nabla^2 g(\xi_{\theta})\theta\right\|\\
         &\le \lambda \sum_{i=1}^{n}\left\|T(y_i)\right\|+\sum_{i=1}^{n} \left\|\nabla c(\theta,x_i)\right\|+\left\|\nabla g(0)\right\|+\lambda_2 \left\|\theta\right\|.        
     \end{align*}
     Now, in the case, where $\nabla c(\theta,x)$ is bounded, we have
     \begin{align*}
          \left\|\nabla \log f(\theta\mid y,x)\right\|&\le \lambda \sum_{i=1}^{n}\left\|T(y_i)\right\|+n\, K(x_1,x_2\cdots,x_n)+\left\|\nabla g(0)\right\|+\lambda_2 \left\|\theta\right\|.
     \end{align*}
     And, in the case, where $\nabla^2 c(\theta,x)$ is bounded, we have 
     \begin{align*}
             \left\|\nabla \log f(\theta\mid y,x)\right\|&\le \lambda \sum_{i=1}^{n}\left\|T(y_i)\right\|+n\,\max_{1\le i\le n}\left\|\nabla c(0,x_i)\right\|+n\, K(x_1,x_2,\cdots,x_n)+\left\|\nabla g(0)\right\|+\lambda_2 \left\|\theta\right\|.      
     \end{align*}
  In the case $g(\cdot)$ is strongly convex, we have \( f_s(\|\theta\|)=\lambda_1 \|\theta\|\)
 and in the case where $g(\cdot)$ is dissipative, we have
 \(f_s(\|\theta\|)=a_{\dag}\|\theta\|\).
Note that when under the condition that 
\begin{align*}
    \left\|\theta\right\|\ge \frac{1}{\gamma-\eta \lambda_2} \left(C_1+\lambda \sum_{i=1}^{n} \left\|T(y_i)\right\|+\tilde{J}(x_1,x_2,\cdots,x_n)+\left\|\nabla g(0)\right\|\right)
\end{align*}
it immediately follows that 
\[f_s(\left\|\theta \right\|)-C_1 \ge \eta \left\|\nabla \log f(\theta \mid y,x)\right\|\]
for both cases with $\eta=1/(\gamma-\eta \lambda_2)$
\end{proof}
\begin{remark}\label{rmrkMp}
    Lemma~\ref{glmmplemma} is important as it allows one to establish a value for $M_p'$ which in turn allows us to obtain a rate of convergence of the algorithm.
\end{remark}
\begin{lemma}\label{glmrepslemma}
    Let Assumptions~\ref{assm:glm:diff}-~\ref{assm:prior:Lipschitz} hold and let either Assumption~\ref{assm:prior:cnvx} or Assumption~\ref{assm:prior:diss} be true. Consider a RWMH algorithm with proposal satisfying assumption~\ref{assm:prop} with stationary distribution as the distribution corresponding to the density given by~\eqref{glmposterior}. Consider any $\epsilon>0$ with $K_{\epsilon}$ such that $Q(0,B(0,K_{\epsilon}))<\epsilon$ and 
    \[\delta <\frac{\epsilon}{q(0) \, C_B(p)\, K^{p-1}_{\epsilon}}\]
    where $C_B(p)$ is the volume of the unit ball. In such a setting,
\begin{equation}\label{glmrepsilon}
\begin{split}
\mathcal{R}_{\epsilon}&=\max\Bigg\{\frac{1}{\gamma}\left(C_1-\frac{1}{\delta}\log\epsilon\right)+K_{\epsilon},\ 2K^2_{\epsilon}\left[\left(\frac{\epsilon}{q(0)\, C_B(p) \delta}\right)^{1/p-1} -K_{\epsilon}\right]^{-1}+K_{\epsilon},\\
& \quad \quad \quad \frac{1}{\gamma}\left(C_1-\log \left(\exp\left\{-f_{x,y}\left(\frac{1}{\gamma}\left(C_1\vee 1\right)\right)\right\}/p^*\wedge 1\right)\right)+1, \max\{M'_p,1\}+K_{\epsilon}\vee1\Bigg\}.
\end{split}
\end{equation}
where $C_1$ is as defined in \eqref{eq:C_1_1}-\ref{eq:C_1_4} , $M'_p$ is as defined in Lemma~\ref{glmmplemma} and  $f_{x,y}$ is as defined in Lemma~\ref{lemmabounded1}.
\end{lemma}
\begin{proof}
     The proof follows immediately from  Lemma~\ref{glmsuplem} Lemma~\ref{lemmabounded1}, Lemma~\ref{glmmplemma} and Proposition~\ref{prop:R_eps} in that order.
\end{proof}

\begin{lemma}\label{glmr1lemma}
      Let Assumptions~\ref{assm:glm:diff}-~\ref{assm:prior:Lipschitz} hold and let either Assumption~\ref{assm:prior:cnvx} or Assumption~\ref{assm:prior:diss} be true. Consider a RWMH algorithm with proposal satisfying assumption~\ref{assm:prop} with stationary distribution as the distribution corresponding to the density given by~\eqref{glmposterior}. For such an algorithm, we have 
      \[\mathcal{R}_{\max}=f_s^{-1}\left(C_1-\log\left(\frac{(1-\tilde\lambda)^2}{2\,p^* \,(2b+\epsilon_{\alpha})^2}\wedge 1\right)\right)\vee \max\{M'_p,1\}\]
      where $\epsilon_{\alpha}<\eta$ as defined in Lemma~\ref{glmmplemma}, 
      \[b=3\, \exp\left(\frac{1}{2}f_{x,y}(\mathcal{R}_{\epsilon})\right)\]
      where $\mathcal{R}_{\epsilon}$ is as defined in \eqref{glmrepsilon}, $M'_p$ is as defined in Lemma~\ref{glmmplemma}, $\tilde\lambda$ is as defined in \eqref{eq:driftcoeff} and $C_1$ is as given in Lemma~\ref{glmsuplem}.
\end{lemma}
\begin{proof}
         The proof follows immediately from Lemma~\ref{glmsuplem}, Lemma~\ref{lemmabounded1}, Lemma~\ref{glmmplemma},  Lemma~\ref{glmrepslemma} and Proposition~\ref{prop:R_eps} in order.
\end{proof}
\subsection{Proof of Theorem~\ref{logisticthm}}
\label{app:logisticthem}

\begin{proof}[Proof of Theorem~\ref{logisticthm}]
    Note that the posterior mentioned has the same form as the posterior of the family \eqref{glmfam} and \eqref{glmprior}.
Here $\Pi(\theta,x_i)=x^{T}_i\theta$ and $T(y_i)=y_i$. Hence $$\left\|\sum_{i=1}^{n}\nabla \Pi(\theta,x_i)T(y_i)\right\|=\left\|\sum_{i=1}^{n}y_i x_i\right\|\le \max_{1\le i\le n}\left\|x_i\right\|\sum_{i=1}^{n}\left\|y_i \right\|=\lambda(x_1,x_2,\cdots,x_n)\sum_{i=1}^{n}\left\|y_i \right\|.$$
We also have $c(\theta,x_i)=\log(1+e^{x^{T}_i\theta})$. Hence, $\nabla c(\theta,x_i)=\frac{e^{x^{T}_i\theta}}{1+e^{x^{T}_i\theta}}x_i$. This implies $$\left\|\sum_{i=1}^{n}\nabla c(\theta,x_i)\right\|\le n\max_{1\le i\le n} \|x_i\|.$$ Now using the notation as in ~\eqref{glmprior}, $g(\theta)=\frac{1}{2}\|\theta\|^2$.
So
\begin{align*}
\nabla^2 g(\theta)&=I.
\end{align*}
This implies that we can invoke Proposition~\ref{prop1} to obtain an upper bound for the rate of convergence of the RWMH algorithm. For this, we need to calculate all the quantities involved in the rate. 
Given the problem we have 
\begin{align*}
    K_1&=0\\
    K_2&=\left(n+\sum_{i=1}^{n}\left\|y_i \right\|\right)\max_{1\le i\le n}\left\|x_i\right\|\\
    K_3&=1.
\end{align*}
This immediately implies
\begin{align*}
    f_{x,y}(u)=\max_{1\le i\le n}\left\|x_i\right\|\left(n+\sum_{i=1}^{n}\left\|y_i \right\|\right)\, u+u^2.
\end{align*}
Also 
\begin{align*}
\lambda_1&=1\\
\lambda_2&=1+\frac{1}{4}\lambda_{max}(X^{\mathsf{T}}X)\\
    C_1&=\left(n+\sum_{i=1}^{n}\left\|y_i \right\|\right)\max_{1\le i\le n}\left\|x_i\right\|\\
    f_s(x)&= \left(1+\lambda_{max}(X^{\mathsf{T}}X)/4\right)x\\
    \delta &=\frac{\epsilon}{2\,q(0)\,C_B(p)\, K^{p-1}_{\epsilon}}
\end{align*}
where $K_{\epsilon}$ is such that $Q(0,B(0,K_{\epsilon}))<\epsilon$. Using these quantities,
\begin{align*}
    M'_p&=\max\left[\frac{1}{1-\eta\left(1+\lambda_{max}(X^{\mathsf{T}}X)/4\right)}\left(C_1\, + \max_{1\le i\le n}\left\|x_i\right\|\sum_{i=1}^{n}\left\|y_i\right\|+n \right),1\right]\\
    \mathcal{R}_{\epsilon}&=\max\Bigg\{C_1-\frac{1}{\delta}\log \epsilon +K_{\epsilon}, 2\, K_{\epsilon}\left[2^{1/(p-1)}-1\right]^{-1}+K_{\epsilon},\\
    &\quad \quad C_1+f_{x,y}(C_1\vee 1)+\log(p^*\wedge 1)+1,1+K_{\epsilon}\vee 1,M'_p\Bigg\}\\
    b&=3\, \exp\left(\frac{1}{2}f_{x,y}(\mathcal{R}_{\epsilon})\right)\\
    \mathcal{R}_{\max}&=\left[C_1-\log\left(\frac{1-\tilde{\lambda}}{2\,p^*\,\left(2b+\epsilon_{\alpha}\right)^2}\right)\wedge 1\right]\vee 1
\end{align*}
where $p^*$ is the the maximum of $f(\theta \mid x,y)$ and $\tilde{\lambda}$ and $\epsilon_{\alpha}$ are as defined in \eqref{eq:driftcoeff} and assumption~\ref{assm:cone} respectively.

Hence using Proposition~\ref{prop1}, we have the RWMH algorithm convergence rate bounded by 
\begin{align*}
    \rho=\max\left\{\left(1-\tilde{\eta}\right)^r,\alpha^{-(1-r)}\right\}A^r
\end{align*}
where 
\begin{align*}
    \tilde{\eta}&=\exp\left\{-2\,\left(K_1+K_2\, \mathcal{R}_{\max}+K_3\,\mathcal{R}_{\max}^2\right)\right\}\, q(\mathcal{R}_{\max}) \mu^{Leb}\left(B(0,\mathcal{R}_{\max})\right)\\
	 A&=1+\frac{4\tilde\lambda b+2\tilde\lambda\epsilon_{\alpha}}{1-\tilde{\lambda}}+2b=\frac{2\tilde\lambda\left(b+\epsilon_{\alpha}\right)+2b}{1-\tilde{\lambda}},\\
 \tilde\alpha&=\left(1+\frac{2b+\epsilon_{\alpha}}{1-\tilde{\lambda}}\right)\left[1+2b+\frac{\tilde{\lambda}}{1-\tilde{\lambda}}\left(2b+\epsilon_{\alpha}\right)\right]^{-1}.
\end{align*}
\end{proof}

\section{Proofs of results for Section~\ref{sec:spectral_results}}
\begin{lemma}\label{prop:conductance}
	Under assumptions~\ref{spectral:strngcnvx} and~\ref{spectral:proposalassm}, for the RWMH Markov kernel $P$, wih stationary distribution $F$,
	\begin{align*}
	\Phi^*_P \le \frac{\left(2\pi\right)^{p/2} e^{-V_{\dag}(0)}}{\left(m+m_1\right)^{p/2}}.
	\end{align*}
\end{lemma}
\begin{proof}[Proof of Lemma~\ref{prop:conductance}]
	Denote by $\alpha(x)$ the acceptance from $x$, that is 
 \[\alpha(x)=\int \alpha(x,,y) q(x,y) dy.\]
 We note that 
	\begin{align*}
	U(y)-U(x) \ge \left\langle \nabla U(x),y-x\right\rangle +\frac{m}{2}\left\|y-x\right\|^2
	\end{align*}
	by assumption~\ref{spectral:strngcnvx}.
	Therefore 
	\begin{align*}
	\alpha(x) &\le \int \min\left\{1,\exp\left(\left\langle \nabla U(x),y-x\right\rangle +\frac{m}{2}\left\|y-x\right\|^2\right)\right\}\, q\left(\left\|y-x\right\|\right) dy\\
	&\le \int \exp\left(\left\langle \nabla U(x),y-x\right\rangle +\frac{m}{2}\left\|y-x\right\|^2-V_{\dag}(\|y-x\|)\right)dy
	\end{align*}
	Also, by assumption~\ref{spectral:proposalassm},
	\begin{align*}
	V_{\dag}(s)\ge V_{\dag}(0)+\frac{s^2}{2}\,m_1.
	\end{align*}
	Note that the $V_{\dag}'(\cdot)$ term is not present as $0$ is the minima for $V_{\dag}(\cdot)$ and hence $V_{\dag}'(0)=0$. Therefore
	\begin{align*}
	\alpha(x) &\le \int \exp\left(-\frac{\left(m+m_1\right)}{2}\left\|y-x\right\|^2-\left\langle  \nabla U(x), y-x\right\rangle-V_{\dag}(0) \right) dy\\
	& \le \exp(-V_{\dag}(0)) \int \exp\left(-\frac{\left(m+m_1\right)}{2}\left\|y-x\right\|^2-\left\langle  \nabla U(x), y-x\right\rangle\right) dy\\
	& \le \exp\left\{-V_{\dag}(0)+\frac{\left|\nabla U(x)\right|^2}{2\,\left(m_1+m\right)}\right\}\, \int \exp \left\{-\frac{1}{2}\left\|\sqrt{\left(m+m_1\right)} \left(y-x\right)+\frac{\nabla U(x)}{\sqrt{\left(m_1+m\right)}}\right\|^2\right\}\, dy.
	\end{align*}
	Using the change of variable 
	\[z=\sqrt{\left(m+m_1\right)} \left(y-x\right)+\frac{\nabla U(x)}{\sqrt{\left(m_1+m\right)}}\]
	\begin{align*}
	\alpha(x) &\le \exp\left\{-V_{\dag}(0)+\frac{\left\|\nabla U(x)\right\|^2}{2\,\left(m_1+m\right)}\right\} \int \left(m+m_1\right)^{-p/2}\, \exp\left(-\frac{\left\|z\right\|^2}{2}\right)\, dz\\
	&= \frac{\left(2\pi\right)^{p/2}e^{-V_{\dag}(0)}}{\left(m+m_1\right)^{p/2}}\, \exp\left\{\frac{\left\|\nabla U(x)\right\|^2}{2\,\left(m+m_1\right)}\right\}.
	\end{align*}
	Define 
	\[B_{\rho_a}=\left\{x: \frac{\left\|\nabla U(x)\right\|}{\sqrt{\,\left(m+m_1\right)}}\le \rho_a\right\}.\]
	Note that $\rho_a$ can be chosen small enough such that $F(B_{\rho_a}) \le F(B^c_{\rho_a})$. This implies 
	\begin{align*}
	\Phi^*_P &\le \frac{1}{F(B_{\rho_a})}\int_{B_{\rho_a}} F(dx)\, P(x,B^c_{\rho_a})\\
	& \le \frac{1}{F(B_{\rho_a})}\int_{B_{\rho}} F(dx)\, \alpha(x)\\
	&\le \frac{\left(2\pi\right)^{p/2}e^{-V_{\dag}(0)}}{\left(m+m_1\right)^{p/2}}\, \exp\left(\frac{1}{2}\rho_a^2\right).
	\end{align*}
	Therefore taking $\rho_a\to 0^+$, we have the result.
\end{proof}
\begin{lemma}\label{prop:spectral:gap}
	Under assumptions~\ref{spectral:strngcnvx} and ~\ref{spectral:proposalassm0}, 
	\begin{align*}
	Gap_R(P)\le \frac{1}{2}\,L\,\frac{p\,e^{-V_{\dag}(0)}} {\tilde{J}^{p/2+1}}
	\end{align*}
	where $\tilde{J}=m_1$ in the first case and $\tilde{J}=V_{\dag}''(0)$ in the second case.
\end{lemma}
\begin{proof}[Proof of Lemma~\ref{prop:spectral:gap}]
	Note that by Cramer-Rao lower bound, for any vector $v \in \mathbb{R}^p$, 
	\[Var_F\left(\left\langle v,X\right\rangle \right)\ge v^{\mathsf{T}}\, \mathbb{E}_F\left[\nabla^2 U(X)\right]^{-1}\, v\ge \left\|v\right\|^2\, L^{-1}.\]
	Also, for the function $g_v(x)=\langle v, x-\mathbb{E}_F(X)\rangle$, 
	\begin{align*}
	\mathcal{E}(P,g_v)&=\frac{1}{2}\int F(dx)\, P(x,dy)\left(g_v(y)-g_v(x)\right)^2\\
	&=\frac{1}{2}\int F(dx)\, P(x,dy) \left\langle v,y-x\right\rangle^2\\
	&\le \frac{1}{2} \int F(dx)\, q\left(\left\|y-x\right\|\right)\, \left\langle v,y-x\right\rangle^2 \, dy \\
	&\le \frac{1}{2}\int F(dx) \left\|v\right\|^2 \left\|z\right\|^2 Q(0,dz)\\
	&\le \frac{\left\|v\right\|^2}{2} \int \left\|z\right\|^2 q(\|z\|) dz.
	\end{align*}
	Now, note that
	\begin{align*}
	\int \left\|z\right\|^2 q(\|z\|) dz &= \int \left\|z\right\|^2 \exp\left(-V_{\dag}(\|z\|)\right) dz\\
	&=\int \|z\|^2 \exp\left(-V_{\dag}(0)-\frac{\|z\|^2}{2}\int_{0}^{1} V_{\dag}''(s\|z\|) ds\right).
	\end{align*}
	where the second line follows from the fact that $0$ is the mode of $q(\cdot)$. Now, in the case there exists an $m_1>0$, 
	\begin{align*}
	\int \left\|z\right\|^2 q(\|z\|) dz &\le  \int \|z\|^2 \exp\left(-V_{\dag}(0)-\frac{m_1\, \|z\|^2}{2}\right)\\
	&=\frac{p\, e^{-V_{\dag}(0)}}{m^{p/2+1}_1}.
	\end{align*}
	In the case where $V_{\dag}'''(\cdot)\ge 0$, we know that 
	\[\frac{\|z\|^2}{2}\int_{0}^{1} V_{\dag}''(s\|z\|) ds\ge \frac{\|z\|^2}{2} V_{\dag}''(0).\]
	Hence we obtain the same bound as done previously with $V_{\dag}''(0)$ instead of $m_1$. Therefore 
	\begin{align*}
	Gap_R(P)&=\inf_{h\in \mathbb{L}_{2,0}(F)}\frac{\mathcal{E}\left(Ph,h\right)}{\left\|h\right\|^2_2}\\
	&\le \frac{\mathcal{E}\left(Pg_v,g_v\right)}{\left\|g_v\right\|^2_2}\\
	&\le \frac{1}{2}\,L\,\frac{p\,e^{-V_{\dag}(0)}} {\tilde{J}^{p/2+1}}
	\end{align*}
	where $\tilde{J}=m_1$ in the first case and $\tilde{J}=V_{\dag}''(0)$ in the second case.
\end{proof}
\end{appendices}

\end{document}